\documentclass[11pt]{amsart}

\usepackage[english]{babel}
\usepackage[dvipsnames]{xcolor}
\usepackage{graphicx}
\usepackage[margin=1in]{geometry} 
\usepackage{amsmath,amsfonts,amsthm,amssymb,amscd,mathrsfs,tikz-cd,adjustbox}
\usepackage{chngcntr}
\usepackage{apptools}
\usepackage{color, soul}
\usetikzlibrary{matrix} 

\usepackage{hyperref}
\hypersetup{pdftoolbar=true, pdffitwindow=true, colorlinks=true, citecolor=blue, filecolor=black, linkcolor=Violet, urlcolor=BlueViolet, hypertexnames=false}

\newtheorem{thm}{Theorem}[section]
\newtheorem{lem}[thm]{Lemma}
\newtheorem{prop}[thm]{Proposition}
\newtheorem{cor}[thm]{Corollary}
\newtheorem{conj}[thm]{Conjecture}
\newtheorem{defn}[thm]{Definition}

\theoremstyle{definition}

\newtheorem{exmp}[thm]{Example}

\newtheorem{rmk}[thm]{Remark}
\newtheorem{notation}[thm]{Notation}

\newtheorem{innercustomthm}{Theorem}
\newenvironment{customthm}[1]
  {\renewcommand\theinnercustomthm{#1}\innercustomthm}
  {\endinnercustomthm}
 
\newcommand{\thistheoremname}{}
\newtheorem*{genericthm*}{\thistheoremname}
\newenvironment{namedthm*}[1]
  {\renewcommand{\thistheoremname}{#1}%
   \begin{genericthm*}}
  {\end{genericthm*}}

\DeclareMathOperator{\Hom}{Hom}
\DeclareMathOperator{\End}{End}
\DeclareMathOperator{\Stab}{Stab}

\numberwithin{equation}{section}

\begin{document}

\title{On The Hecke Orbit Conjecture for PEL Type Shimura Varieties}

\author{Luciena X. Xiao}
\address{Department of Mathematics, Caltech, MC 253-37, 1200 East California Boulevard, Pasadena, CA
91125.
}
\curraddr{}
\email{xxiao@caltech.edu}
%\thanks{}

\setcounter{tocdepth}{1}

\date{}

\date{}

\dedicatory{}

\begin{abstract}
The Hecke orbit conjecture asserts that every prime-to-$p$ Hecke orbit in a Shimura variety is dense in the central leaf containing it. In this paper, we prove the conjecture for certain irreducible components of Newton strata in Shimura varieties of PEL type A and C, when $p$ is an unramified prime of good reduction. Our approach generalizes Chai and Oort's method for Siegel modular varieties. 
\end{abstract}

\maketitle

\tableofcontents

\section{Introduction} 
\subsection{Motivation and history.}

The Hecke orbit conjecture predicts that Hecke symmetries characterize the central foliation on Shimura varieties over an algebraically closed field $k$ of characteristic $p$. The conjecture predicts that on the mod $p$ reduction of a Shimura variety, any prime-to-$p$ Hecke orbit is dense in the central leaf containing it. In the 90s, Oort came up with the idea of studying the locus defined by fixing a geometric isomorphism type of Barsotti-Tate groups. Such a locus is called an Oort's central leaf. By definition, a central leaf is stable under prime-to-$p$ Hecke correspondences and naturally contains the prime-to-$p$ Hecke orbit of any point in it. On the other hand, an isogeny leaf is an orbit by geometric $p$-isogenies. Oort showed that central leaves and isogeny leaves almost give a product structure on an irreducible component of a Newton polygon stratum. Moreover, central leaves define a partition of its ambient Newton stratum by (possibly infinitely many) locally-closed, smooth subvarieties. Within any given Newton stratum, all central leaves have the same dimension and are related to each other via finite isogeny correspondences. In this sense, central leaves form a ``foliation" on a given Newton stratum. So do isogeny leaves. Therefore, isogeny leaves and central leaves lie transversely to each other and characterize the geometry of the Newton stratum they are in.

Given any closed geometric point on a Shimura variety, its orbit under prime-to-$p$ Hecke correspondences naturally sits inside the central leaf passing through that point. The Hecke orbit conjecture draws a parallel between central leaves and isogeny leaves: just as each isogeny leaf coincides with a fixed geometric $p$-isogeny class, central leaves should also almost coincide with a fixed prime-to-$p$ isogeny class. In \cite[Problem 18]{MR1812812}, Chai and Oort predicted that prime-to-$p$ Hecke orbits on Siegel modular varieties are as large as possible. That is to say, every prime-to-$p$ Hecke orbit is Zariski dense in the central leaf containing it. This also means that central leaves are minimal subvarieties stable under all prime-to-$p$ Hecke correspondences. Here we state the more general version of the conjecture for PEL type Shimura varieties.

\begin{conj}\cite[Conjecture 3.2]{chai2006hecke}\label{conj}
Every prime-to-$p$ Hecke orbit on a PEL type Shimura variety over $k$ is dense in the central leaf containing it.
\end{conj}

In 1995, Chai proved the conjecture for ordinary points (\cite[Theorem 2]{chai1995every}) on Siegel modular varieties. In 2019, following a strategy similar to that in \cite{chai1995every}, Rong Zhou (\cite[Theorem 3.1.3]{zhou2019motivic}) proved the Hecke orbit conjecture for the $\mu$-ordinary locus on quaternionic Shimura surfaces and their associated unitary Shimura varieties. 

The first known case of the full conjecture is proven by C.-F. Yu for Hilbert modular varieties \cite{yu2004discrete}. Using the statement of Hilbert modular varieties, Chai and Oort proved in 2005 that the conjecture holds for Siegel modular varieties (\cite[Theorem 13.1]{chai2005hecke}).  

The present paper concerns the case of Shimura varieties of PEL type. They are moduli spaces of abelian varieties in characteristic $p$ with prescribed additional structures: polarization, action by a finite dimensional $\mathbb{Q}$-algebra, and level structure. We generalize the method in \cite{chai2005hecke} to applicable situations for PEL type Shimura varieties.

\subsection{Overview of results.}
We fix an integer prime $p$ and let $k$ be an algebraically closed field of characteristic $p.$ Let $\mathcal{D}=(B,\mathcal{O}_B,*,V,(\cdot,\cdot), h)$ be a Shimura datum of PEL type (see \ref{2.1} and \cite{kottwitz1992points} for details) for which $p$ is an unramified prime of good reduction. Write $F$ for the center of $B$ and $F_0$ the maximal subfield fixed under $*$. 

The main result of this paper is the following theorem. We refer the readers to Theorem \ref{mainthm} for the precise statement.

\begin{thm}
Let $\mathscr{S}$ be the reduction modulo $p$ of a Shimura variety of PEL type A or C over $k$ for which $p$ is an unramified prime of good reduction. Let $\mathcal{N}$ be a Newton stratum. Assume \begin{enumerate}
    \item $\mathcal{N}$ contains a $B$-hypersymmetric point $x_0$;
    \item either $p$ is totally split in $F/F_0$ and the Newton polygon of $\mathcal{N}$ satisfies condition (*), or $p$ is inert in $F/F_0$.
\end{enumerate}
Write $\mathcal{N}^0$ for the irreducible component of $\mathcal{N}$ containing $x_0.$ Then $H^p(x)$ is dense in $C(x)\cap\mathcal{N}^0$ for every $x\in\mathcal{N}^0(k).$ Moreover, if $\mathcal{N}$ is not the basic stratum, then $C(x)\cap \mathcal{N}^0$ is irreducible.
\end{thm}

Assumption (2) only occurs in the case of PEL type A. Condition (*) (see Definition \ref{conditionstar}) amounts to a mild condition on the slope data of the Newton polygon attached to $\mathcal{N}.$ Condition (*) is only necessary for proving the main theorem for points that are not $B$-hypersymmetric when $B$ is not a totally real field. Theorem \ref{irred}, Theorem \ref{cts}, and Proposition \ref{B=F} are independent of this assumption.

The condition that $\mathcal{N}$ is not the basic stratum has to do with a monodromy result (Theorem \ref{thmk}) used in proving $C(x)\cap \mathcal{N}^0$ is geometrically irreducible (or equivalently, geometrically connected). A straightforward consequence of Theorem \ref{thmk} is that, if we further assume $\mathcal{N}$ is smooth, then the assumption that $\mathcal{N}$ is geometrically irreducible is equivalent to the condition that $H^{\Sigma}$ (see section \ref{section 4}) acts transitively on the set of geometrically irreducible components of $\mathcal{N}$. We obtain the following corollary.

\begin{cor}
Assumptions as in Theorem 1.2. Further assume that $N$ is smooth and that the prime-to-$\Sigma$ Hecke correspondences act transitively on $\pi_0(\mathcal{N})$. Then the Hecke orbit conjecture holds for $\mathcal{N}$.
\end{cor}

A key input in the proof strategy is $B$-hypersymmetric points. Hypersymmetric abelian varieties are mod-$p$ analogues of CM abelian varieties in characteristic $0.$ Originally developed by Chai and Oort (see \cite{chai2006hypersymmetric}), the notion refers to abelian varieties that admit as many endomorphisms as allowed by their associated Barsotti-Tate groups. In Section \ref{section 3}, we discuss the details regarding the existence of $B$-hypersymmetric points in relation to the shape of Newton polygons. In Section \ref{mu}, we restrict our attention to the $\mu$-ordinary locus in Shimura varieties of PEL type A. We deduce two sufficient conditions for a Newton stratum to contain a $B$-hypersymmetric point. These conditions combined with the main theorem imply the following (notations as in Section \ref{2.1}).

\begin{cor}[Corollary \ref{cormu}]\begin{enumerate}
    \item Suppose $p$ is inert in $F$. If every slope of the Newton polygon attached to $\mathcal{N}$ has the same multiplicity, then the Hecke Orbit conjecture holds for any irreducible component of $\mathcal{N}$ containing a $B$-hypersymmetric point.
    \item Suppose the center of $B$ is a CM field. Assume that the signature of $\mathscr{S}$ has no definite place, and that $p$ is a prime of constant degree in the extension $F/\mathbb{Q}$. Further assume assumption 2 of the main theorem is satisfied. Then the Hecke orbit conjecture holds for every irreducible component of the $\mu$-ordinary stratum.
\end{enumerate}
\end{cor}

The statement of our main theorem is restricted to irreducible components of Newton strata which contain $B$-hypersymmetric points. In general, it is not known whether Newton strata in Shimura varieties of PEL type are irreducible or not. On the other hand, it is well-known that if the basic locus coincides with the supersingular locus, then it is discrete; otherwise, the basic locus may be of positive dimension. 

Oort and Chai proved in \cite[Theorem A]{chai2011monodromy} that every non-supersingular Newton stratum in a Siegel modular variety is irreducible. Siegel modular varieties classify abelian varieties equipped with a principal polarization. However, the statement is not correct if the polarization is not principal (see \cite{chai2011monodromy}). 

For Shimura varieties of PEL type, the only irreducibility result that the author is aware of is \cite[Theorem 1.1]{MR3240772}, where Achter proved for a special class of PEL type Shimura varieties that all Newton strata are irreducible. His result allows us to deduce the following consequence.

\begin{cor}[Corollary \ref{corachter}]
Let $L$ be an imaginary quadratic field inert at the rational prime $p$. The Hecke Orbit conjecture holds for the moduli space of principally polarized abelian varieties of dimension $n\ge 3$ over $k$ equipped with an action by $\mathcal{O}_L$ of signature $(1, n-1).$
\end{cor}

\subsection{Overview of strategy.}
A general strategy for attacking Conjecture \ref{conj} is to break it down into the following two parts as in \cite[Conjecture 4.1]{chai2005hecke}. It is clear that the conjecture is equivalent to their conjunction.

\begin{itemize}
\item The discrete part: the prime-to-$p$ Hecke orbit of $x$ meets every irreducible component of the central leaf $C(x)$ passing through $x$;
\item The continuous part: the Zariski closure of the prime-to-$p$ Hecke orbit of $x$ in $C(x)$ has the same dimension as $C(x)$.
\end{itemize}

In this section, We first give a brief description for Chai and Oort's strategy for proving the Hecke orbit conjecture for Siegel modular varieties in \cite{chai2005hecke}. Then we highlight the differences as well as new ideas in our approach.  

For the discrete part, Chai and Oort proved a stronger result that every central leaf not contained in the supersingular stratum is irreducible. This is a consequence of the fact that on a Siegel modular variety, every non-supersingular Newton stratum on a Siegel modular variety is irreducible (\cite[Theorem A]{chai2011monodromy}), and every Newton stratum contains a hypersymmetric point (\cite[Theorems 5.4]{chai2006hypersymmetric}). 

For the continuous part, one analyzes the formal completion of $C(x)$ at a split hypersymmetric point. A point is called split if the corresponding abelian variety is isogenous to a product of abelian varieties with at most two slopes. It turns out that the formal completion of the Zariski closure of the prime-to-$p$ Hecke orbit of a split hypersymmetric point $y$ in $C(x)$ is a formal Barsotti-Tate subgroup of the formal completion of $C(x)$. Furthermore, the action of the local stabilizer group of $y$ on $C(x)^{/x}\cong C(x)^{/y}$ underlies an absolutely irreducible representation. Thus the conjecture is true for any split hypersymmetric point. To deduce the statement for arbitrary geometric points, one uses the Hecke orbit conjecture for Hilbert modular varieties to find a split hypersymmetric point in the interior of the closure of any prime-to-$p$ Hecke orbit in its central leaf. This completes the proof in the Siegel case.

Shimura varieties of PEL type are subvarieties of Siegel modular varieties cut out by the condition of having an action by the ring of integers in a prescribed finite dimensional semisimple $\mathbb{Q}$-algebra $B$. Although many of the ingredients used in \cite{chai2005hecke} are known for PEL type Shimura varieties, there are two major things that do not work the same way.

First, $B$-hypersymmetric points on PEL type Shimura varieties are not as abundant as in the Siegel case. Not every Newton stratum contains a hypersymmetric point (see \cite[Example 5.3]{zong2008hypersymmetric} and Example \ref{EmptyExample}). A rephrase of Zong's main result \cite[Theorem 5.1.1]{zong2008hypersymmetric} says that a Newton stratum contains a $B$-hypersymmetric point if and only if the associated Newton polygon is $B$-symmetric (see Theorem \ref{hs}).

Secondly, Chai and Oort's approach depends on the Hilbert trick, which refers to the property that every point on a Siegel modular variety comes from a Hilbert modular variety (see \cite[Section 9]{chai2005hecke}). The application of this fact is two-fold: (1) to find a hypersymmetric point inside the closure of every Hecke orbit inside its central leaf and (2) to show that one can take such a hypersymmetric point to be split, thereby reducing the continuous part into the two-slope case. The Hilbert trick is true for PEL type C only, where the Hecke correspondences also come from a symplectic group. For a general simple $\mathbb{Q}$-algebra $B$, we deduce the conjecture under mild conditions from the conjecture for the Shimura variety attached to $F_0,$ the maximal subfield of $B$ fixed under the positive involution $*$ of $B.$ We bypass the second usage of the Hilbert trick by leveraging on the fact that the formal completion of a central leaf on a PEL type Shimura variety admits a cascade structure built up from Barsotti-Tate formal groups. The formal completion of the Zariski closure of a prime-to-$p$ Hecke orbit, as a subscheme of the formal completion of the central leaf, is then determined by its image in the two-slope components of the cascade. We thereby reduce to an analogue of a step in the proof of the Siegel case, establishing the desired equality at the level of two-slope components, from which the continuous part follows by an inductive argument.

\begin{rmk}
We exclude PEL type D in this paper, in which case the algebraic group $G$ is disconnected. We expect to extend the method to cover case D with extra work in the future.
\end{rmk}

\begin{rmk}
Sug-Woo Shin \cite{shin} announced a proof for the irreducibility of central leaves on Shimura varieties of Hodge type. His proof relies on the theory of automorphic forms, and, unlike Chai and Oort's approach, is independent of the irreducibility of Newton strata. Since our proof of the continuous part also does not depend on the irreducibility of Newton strata, Shin's result combined with Theorem 1.2 will yield the following.
\end{rmk}

\begin{thm}
Let $\mathscr{S}$ be the reduction modulo $p$ of a Shimura variety of PEL type A or C over $k$ for which $p$ is an unramified prime of good reduction. Let $\mathcal{N}^0$ be an irreducible component of a Newton stratum. 
\begin{enumerate}
    \item Assume $\mathcal{N}^0$ contains a $B$-hypersymmetric point $x_0$;
    \item either $p$ is totally split in $F/F_0$ and the Newton polygon of $\mathcal{N}^0$ satisfies condition (*), or $p$ is inert in $F/F_0$.
\end{enumerate}
Then $H^p(x)$ is dense in $C(x)$ for every $x\in\mathcal{N}^0(k).$
\end{thm}

\subsection{Paper outline.}
We briefly discuss the organization of this paper. Sections 2 recalls definitions and various known facts relevant to our context. Section \ref{section 3} develops new terminologies to describe $B$-hypersymmetric points on PEL type Shimura varieties. We rephrase Zong's main result in \cite{zong2008hypersymmetric} in simpler language and derive conditions on the existence of hypersymmetric points in special cases relevant to our applications. Section \ref{section 4} contains the proof of a monodromy result that serves as a key input in proving the irreducibility of $C\cap\mathcal{N}^0.$ This result generalizes \cite[Theorem 5.1]{chai2005hecke} and the main result of \cite{kasprowitz2012monodromy}. Section 5 contains the proof of the discrete part of the main theorem. In section 6, we prove the continuous part for $B$-hypersymmetric points, which then culminates in section 7 with a reduction argument proving the main theorem at general points. We restrict our attention to Cases A and C in sections 5 and 7.

\subsection{Future directions.}
Some of the tools used for proving the conjecture for PEL case are known in more general settings. For example, the almost product structure of Newton strata and Serre-Tate theory are both generalized to Shimura varieties of Hodge type (see \cite{hamacher2019product}; \cite{hong2019hodge} and \cite{shankar2016serre}). However, the right notion for hypersymmetric points in the Hodge type case remains open.

\hfill

\noindent\textbf{Acknowledgements.} I sincerely thank my advisor Elena Mantovan for her continuous patience and unwavering support. I'm grateful to Ana Caraiani, Serin Hong, Marc-Hubert Nicole, Eric Rains, and Sug Woo Shin for helpful discussions and correspondences.

\section{Preliminaries}
\subsection{Shimura varieties of PEL type.}\label{2.1}
Fix a prime number $p$ throughout the rest of this paper. We are interested in moduli problems of PEL type as given in Kottwitz \cite[\textsection 5]{kottwitz1992points}.

Let $\mathcal{D}=(B,\mathcal{O}_B,*,V,(\cdot,\cdot), h)$ be a Shimura datum of PEL type consisting of the following information: \begin{itemize}
    \item $B$, a finite dimensional simple $\mathbb{Q}$-algebra; 
    \item $\mathcal{O}_B$, a maximal $\mathbb{Z}_{(p)}$-order of $B$;
    \item $*$, a positive involution on $B$ preserving $\mathcal{O}_B$;
    \item $V$, a nonzero finitely generated left $B$-module such that $V_{\mathbb{Q}_p}$ contains a self dual lattice preserved by $\mathcal{O}_B$;
    \item $(\cdot,\cdot)$, a $\mathbb{Q}$-valued nondegenerate hermitian form on $V$ such that $(bv,w)=(v,b^*w)$ for all $v,w\in V$ and all $b\in B;$
    \item a homomorphism $h:\mathbb{C}\rightarrow \End_{B\otimes_{\mathbb{Q}}\mathbb{R}}(V\otimes_{\mathbb{Q}}\mathbb{R})$ such that $(v,w)\mapsto (v,h(\sqrt[]{-1})w)$ is a positive definite symmetric form on $V_{\mathbb{R}}$.
\end{itemize}

Let $F$ denote the center of $B$ and let $F_0$ be the maximal totally real subfield of $F$ fixed under $*.$

We assume in addition that $p$ is an unramified prime of good reduction for the Shimura datum $\mathcal{D}$. Equivalently, $B_{\mathbb{Q}_p}$ is a product of matrix algebras over unramified extensions of $\mathbb{Q}_p$. In particular, $F$ is unramified at $p$.

One associates to the Shimura datum $\mathcal{D}$ an linear algebraic group $G$ such that $G(R)=\{x\in \End_B(V)\otimes_{\mathbb{Q}}R|xx^*\in R^{\times}\}$ for any $\mathbb{Q}$-algebra $R$. The homomorphism $h$ gives a decomposition of the $B_{\mathbb{C}}$-module $V_{\mathbb{C}}$ as $V_{\mathbb{C}}=V_1\oplus V_2$, where $V_1$ (resp. $V_2$) is the subspace on which $h(z)$ acts by $z$ (resp. by $\overline{z}$). $V_1$ and $V_2$ are $B_{\mathbb{C}}$-submodules of $V_{\mathbb{C}}.$ The field of definition of the isomorphism class of the complex representation $V_1$ of $B$, denoted by $E,$ is called the reflex field of $\mathcal{D}.$

Now we describe the following moduli problem associated to the Shimura datum $\mathcal{D}$ as given in \cite[Section 5]{kottwitz1992points}. Let $\mathbb{A}_f^p$ denote the ring of finite adeles attached to $\mathbb{Q}$ with trivial $p$-component. Let $K=K_pK^p\subseteq G(\mathbb{A}_f)$ be a subgroup, with $K_p$ being a fixed hyperspecial maximal compact subgroup of $G(\mathbb{Q}_p),$ and $K^p$ being a compact open subgroup of $G(\mathbb{A}_f^p)$. 

Consider the contravariant functor from the category of locally-Noetherian schemes $S$ over $\mathcal{O}_{E,(p)}:=\mathcal{O}_E\otimes_{\mathbb{Z}}\mathbb{Z}_{(p)}$ that associates to $S$ the set of isomorphism classes of quadruples $(A,\lambda, i, \overline{\eta})$, where \begin{itemize}
\item $A$ is an abelian scheme over $S$;
\item $\lambda$ is a prime-to-$p$ polarization of $A;$
\item $i:\mathcal{O}_B \hookrightarrow \End(A)\otimes_{\mathbb{Z}}\mathbb{Z}_{(p)}$ is a morphism of $\mathbb{Z}_{(p)}$-algebras such that $\lambda\circ i(\alpha^*)=i(\alpha)^{\vee}\circ\lambda$ and $\det(b,\text{Lie}(A))=\det(b,V_1)$ for all $\alpha\in\mathcal{O}_B;$
\item $\overline{\eta}$ is a prime-to-$p$ level $K^p$ structure in the following sense. Let $s$ be a geometric point of $S.$ Denote by $A_s$ the fiber of $A$ over $s$ and let $H_1(A_s,\mathbb{A}_f^p)$ denote its Tate $\mathbb{A}_f^p$-module. A level structure of type $K^p$ on $A$ is a $K^p$-orbit $\overline{\eta}$ of isomorphisms $\eta:V_{\mathbb{A}_f^p}\rightarrow H_1(A_s,\mathbb{A}_f^p)$ as skew-Hermitian $B$-modules such that $\overline{\eta}$ is fixed by $\pi_1(S,s).$
\end{itemize}

Two quadruples $(A,\lambda,i,\overline{\eta})$ and $(A',\lambda',i',\overline{\eta}')$ are said to be isomorphic if there exists a prime-to-$p$ isogeny $f:A\rightarrow A'$ such that $\lambda=rf^{\vee}\circ\lambda'\circ f$ for some positive integer $r\in \mathbb{Z}_{(p)}^{\times}$, $f\circ i=i'\circ f$ and $\overline{\eta'}=f\circ \overline{\eta}.$

When $K^p$ is sufficiently small, this functor is representable by a smooth quasi-projective scheme $\mathcal{S}_{K^p}$ defined over $\mathcal{O}_{E,(p)}$ (see \cite[Section 5]{kottwitz1992points}).

We recall that in the terminologies of Kottwitz, such a moduli problem is said to be of type A if $G$ is unitary, type C if symplectic, and type D if orthogonal. Write $F_0$ for the subfield of $F$ fixed by the positive involution $*$. In the case of type A, $F\neq F_0$. Otherwise $F$ is a totally real field.

As the level $K^p$ varies, the varieties $\mathcal{S}_{K^p}$ form an inverse system that carries a natural action by $G(\mathbb{A}_f^p).$ 

From now on we fix a prime $v$ of $E$ over $p$ with residue field $\kappa$ and denote by $$\mathscr{S}_{K^p}:=\mathcal{S}_{K^p}\otimes\overline{\kappa}$$ the special fiber of $\mathcal{S}_{K^p}$ over $v.$

\subsection{Newton stratification and Oort's foliation.} Let $k$ be an algebraically closed field of characteristic $p.$ Write $W=W(k)$ for the Witt ring of $k$ and $L$ the fraction field of $W.$ By an abuse of notation, we use $\sigma$ to denote both the Frobenius on $k$ and that of $L$ over $\mathbb{Q}_p.$ We define the set $B(G)$ of $\sigma$-conjugacy classes of $G(L)$ by $$B(G)=\{[b]|b\in G(L)\},$$ where $$[b]=\{g^{-1}b\sigma(g)|g\in G(L)\}.$$

An $F$-isocrystal is a finite dimensional vector space $V$ over $L$ with the Frobenius automorphism $F:V\rightarrow V.$ An $F$-isocrystal with $G$-structure is an exact faithful tensor functor $$\text{Rep}_{\mathbb{Q}_p}(G)\rightarrow \text{F-Isoc}(k).$$ 
For each $b\in G(L)$, there is an associated $F$-isocrystal with $G$-structure $N_b$ given by $N_b(W,\rho)=(W_L,\rho(b)(id_W\otimes\sigma))$ (see \cite[\textsection 3.4]{rapoport1996classification}). The isomorphism class of $N_b$ is fixed under $\sigma$-conjugation in $G(L).$ Hence the set $B(G)$ is identified with the set of isomorphism classes of $F$-isocrystals with $G$-structures. 

According to the Dieudonn\'e-Manin classification, the category of $F$-isocrystals is semi-simple, where the simple objects are parametrized by rational numbers called slopes. 

Kottwitz classfied points in $B(G)$ by associating to each $\sigma$-conjugacy class $[b]\in B(G)$ a Newton point and a Kottwitz point (see \cite{kottwitz1985isocrystals}, \cite{kottwitz1997isocrystals}, and \cite{rapoport1996classification}). The set of Newton point admits a partial order $\preceq$.

Let $\mu$ be a conjugacy class of cocharacters of $G.$ To $\mu$ we associate the Newton point $$\overline{\mu}=\frac{1}{r}\sum_{i=0}^{r-1}\sigma^i(\mu),$$where $r$ is an integer such that $\sigma^r$ fixes $\mu.$ An element $[b]\in B(G)$ is said to be $\mu$-admissible if the Newton point corresponding to $[b]$ is less than $\overline{\mu}$. We write $B(G,\mu)$ for the set of $\mu$-admissible elements of $B(G).$ $B(G,\mu)$ is finite with a unique minimum called the basic point, and a unique maximum called the $\mu$-ordinary point (see \cite{kottwitz1997isocrystals}).

From now on, we write $b$ for the conjugacy class $[b]\in B(G).$ For any geometric point $x\in\mathscr{S},$ let $X_x$ be the fiber of the universal Barsotti-Tate group at $x.$ Write $N_x$ for the $F$-isocrystal associated to $X_x,$ then $N_x$ uniquely determines an element $b_x\in B(G_{\mathbb{Q}_p},\mu_{\overline{\mathbb{Q}}_p}).$ The following result is due to Rapoport and Richartz (see \cite[Theorem 3.6]{rapoport1996classification}). 

For $b\in B(G_{\mathbb{Q}_p})$, the set $$\mathscr{S}(\preceq b)=\{x\in\mathscr{S}|b_x\preceq b\}$$is a closed subscheme of $\mathscr{S}$ called the closed Newton stratum attached to $b$. The sets $\mathscr{S}(\preceq b)$ form the Newton stratification of $\mathscr{S}$ by closed subschemes of $\mathscr{S}$ indexed by $b\in B(G_{\mathbb{Q}_p}).$ 

The open Newton stratum attached to $b\in B(G_{\mathbb{Q}_p})$ is defined as $$\mathcal{N}_b=\mathscr{S}(\preceq b)-\cup_{b'\preceq b}\mathscr{S}(\preceq b').$$
It is a locally-closed reduced subscheme of $\mathscr{S}.$ Moreover, the stratum $\mathcal{N}_b$ is non-empty if and only if $b\in B(G_{\mathbb{Q}_p},\mu_{\overline{\mathbb{Q}}_p})$ (see \cite[Theorem 1.6]{viehmann2013ekedahl}). Moreover, in the situations of interest to us, this stratification coincides with the classical Newton stratification determined by the isogeny class of the geometric fibers of the universal Barsotti-Tate group (see \cite[Theorem 3.8]{rapoport1996classification}).

Now we fix a conjugacy class $b$ in $B(G_{\mathbb{Q}_p},\mu_{\overline{\mathbb{Q}}_p})$ and write $\mathbb{X}$ for a corresponding Barsotti-Tate group with $G_{\mathbb{Q}_p}$-structure defined over $\overline{\kappa}$. A Barsotti-Tate group $\mathbb{X}'$ over a field $k'\supset \overline{\kappa}$ is geometrically isomorphic to $\mathbb{X}$ (denoted $\mathbb{X}'\cong_g\mathbb{X})$ if $\mathbb{X}'$ and $\mathbb{X}$ become isomorphic over an extension of $k'.$ The central leaf associated to $\mathbb{X}$ is defined as $$C_{\mathbb{X}}=\{x\in \mathscr{S}|X_x\cong_g \mathbb{X}\},$$ where $X$ stands for the universal Barsotti-Tate group. By definition, we have $C_{\mathbb{X}}\subseteq \mathcal{N}_b.$ Moreover, $C_{\mathbb{X}}$ is a locally-closed smooth subscheme of the Newton stratum $\mathcal{N}_b$ (see \cite{oort2004foliations} and \cite[Proposition 1]{mantovan2005cohomology}). 

For a geometric point $x\in \mathscr{S}(k),$ we say $C(x):=C_{X_x}$ is the central leaf passing through $x.$

On any fixed Newton stratum $\mathcal{N}_b$, central leaves give a foliation structure, called the central foliation. Any closed point $x\in \mathcal{N}_b$ is contained in exactly one central leaf. If $x,x'$ are two points in $\mathcal{N}_b$, there exists a scheme $T$ and finite surjective morphisms $C(x)\twoheadleftarrow T\twoheadrightarrow C(x')$. In particular, $\dim C(x)=\dim C(x')$ (see \cite[\textsection 2 and \textsection 3]{oort2004foliations}).

\subsection{Hecke symmetries and Hecke orbits.}
As $K^p$ varies over all sufficiently small compact open subgroups of $G(\mathbb{A}_f^p),$ the Shimura varieties $\mathscr{S}_{K^p}$ form an inverse system $\varprojlim_{K^p}\mathscr{S}_{K^p}$. If $K_1^p\subseteq K_2^p$ are compact open subgroups of $G(\mathbb{A}_f^p),$ there is an \'etale covering $\mathscr{S}_{K_1^p}\rightarrow \mathscr{S}_{K_2^p}$ given by $(A,\lambda,i,(\overline{\eta})_1)\mapsto (A,\lambda,i,(\overline{\eta})_2),$ where $(\overline{\eta})_i$ denotes the $K_i^p$-orbit of $\eta$ and the map is given by extending $(\overline{\eta})_1$ to the $K_2^p$-orbit. 

The inverse system $\varprojlim_{K^p}\mathscr{S}_{K^p}$ admits a natural right action by $G(\mathbb{A}_f^p).$ For $g\in G(\mathbb{A}_f^p),$ the corresponding map $$\mathscr{S}_{K^p}\rightarrow \mathscr{S}_{g^{-1}K^pg}$$ is given by $$(A,\lambda,i,\overline{\eta})\mapsto (A,\lambda,i,\overline{\eta g}).$$

For a fixed $K^p$, the action of $G(\mathbb{A}_f^p$) induces a family of finite \'etale algebraic correspondences on $\mathscr{S}_{K^p}$ called the prime-to-$p$ Hecke correspondences. Namely, for $g\in G(\mathbb{A}_f^p)$, we have $$\mathscr{S}_{K^p}\xleftarrow{a}\mathscr{S}_{K^p\cap gK^pg^{-1}}\xrightarrow{b}\mathscr{S}_{K^p},$$ where $b$ is the covering map $\mathscr{S}_{K^p\cap gK^pg^{-1}}\rightarrow \mathscr{S}_{K^p}$ induced by the inclusion of $K^p$ into $gK^pg^{-1}\subseteq K^p$, and $a$ is the composition of the covering map for the inclusion $g^{-1}(K^p\cap gK^pg^{-1})g\subseteq K^p$ with the isomorphism between $\mathscr{S}_{K^p\cap gK^pg^{-1}}$ and $\mathscr{S}_{g^{-1}(K^p\cap gK^pg^{-1})g}.$

Let $x\in \mathscr{S}_{K^p}(k)$ be a closed geometric point, and let $\tilde{x}\in \mathscr{S}(k)$ be a geometric point of the tower $\mathscr{S}(k)$ above $x.$ The prime-to-$p$ Hecke orbit of $x$ in $\mathscr{S}_{K^p}(k)$, denoted by $H^p(x),$ is the countable set consisting of all points that belong to the image of $G(\mathbb{A}_f^p)\cdot\tilde{x}$ under the projection $\mathscr{S}(k)\rightarrow \mathscr{S}_{K^p}(k).$ For a prime $l\neq p,$ the $l$-adic Hecke orbit, denoted by $H_l(x)$, is defined to be the projection of $G(\mathbb{Q}_l)\cdot\tilde{x}$ to $\mathscr{S}_{K^p}(k)$, where the action is given via the canonical injection $G(\mathbb{Q}_l)\hookrightarrow G(\mathbb{A}_f^p).$ It is clear from the definition that both $H^p(x)$ and $H_l(x)$ are independent of the choice of $\tilde{x}.$ By definition, $H^p(x)$ sits inside the central leaf $C(x)$ passing through $x.$

\section{Hypersymmetric Abelian varieties}\label{section 3}

Hypersymmetric abelian varieties were first studied by Chai and Oort in \cite{chai2006hypersymmetric} as a tool for proving the Hecke Orbit conjecture for Siegel modular varieties. Y. Zong studied the more general version for PEL type Shimura varieties in his dissertation \cite{zong2008hypersymmetric} and gave necessary and sufficient conditions on the Newton polygon for the existence of simple hypersymmetric points. While the existence of such points has applications in proving the irreducibility of central leaves and Igusa towers \cite[Proposition 3.3.2]{eischen2017p}, hypersymmetric points are also of their own interests as mod-$p$ analogues of CM abelian varieties in characteristic $0$.

Recall that an abelian variety $A$ of dimension $g$ over a field $K$ is said to be of CM-type if $\End(A)\otimes_{\mathbb{Z}}\mathbb{Q}$ contains a semi-simple commutative sub-algebra of rank $2g$ over $\mathbb{Q}.$ In a moduli space of abelian varieties over a field of characteristic zero, points of CM-type are special. However, Tate proved that if the base field is of positive characteristic, every abelian variety is of CM-type \cite{tate1966endomorphisms}. In this sense, CM-type abelian varieties are no longer special. 

Notations as in Section \ref{2.1}. Fixing a level structure $K^p$, we may consider geometric points in the moduli space $\mathscr{S}:=\mathscr{S}_{K^p}$ which correspond to abelian varieties that have as many endomorphisms as allowed by their Barsotti-Tate groups. As it turns out, such points are indeed special in the positive characteristic setting. Not every point satisfy this condition (see \cite[Remark 2.4]{chai2006hypersymmetric}). Moreover, in Shimura varieties of PEL type, not every Newton stratum contains such a point.

\begin{defn}\begin{enumerate}
    \item \cite[Definition 6.4]{chai2006hypersymmetric}
A $B$-linear polarized abelian variety $A$ over $k$ is $B$-hypersymmetric if $$\End_B(A)\otimes_{\mathbb{Z}}\mathbb{Q}_p\cong \End_B(A[p^{\infty}])\otimes_{\mathbb{Z}_p}\mathbb{Q}_p.$$
    \item We say a point $x\in \mathscr{S}(k)$ is $B$-hypersymmetric if the corresponding abelian variety $A_x$ is $B$-hypersymmetric. 
\end{enumerate}
\end{defn}

\subsection{New formulation of the characterization of $B$-hypersymmetric abelian varieties.} 
Given a Shimura variety $\mathscr{S}:=\mathscr{S}_{K^p}$ over $k,$ we are interested in the existence of $B$-hypersymmetric points in any central leaf or prime-to-$p$ Hecke orbit. On Siegel modular varieties, a Newton stratum contains a hypersymmetric point if and only if its Newton polygon is symmetric \cite[Corollary 4.8]{chai2006hypersymmetric}. This is true for every Newton stratum on a Siegel modular variety due to the presence of polarization. For PEL type Shimura varieties, Zong showed in \cite{zong2008hypersymmetric} that the existence of $B$-hypersymmetric points depends only on the shape of the corresponding Newton polygon, i.e. a Newton stratum contains $B$-hypersymmetric points precisely when the corresponding Newton polygon satisfies ``supersingular restriction" (see \cite[Theorem 5.1]{zong2008hypersymmetric}). We remark that for the purpose of our paper, the conditions that the Newton polygon of a non-empty Newton stratum admits a ``CM-type partition" (see \cite[Definition 4.1.8]{zong2008hypersymmetric}) is redundant due to the non-emptiness. Hence, for convenience, in the present paper we develop an easier language to describe the shape of Newton polygons whose strata contain $B$-hypersymmetric points, although it is not necessary to do so.

For Shimura varieties of PEL type, we introduce an analogue of the notion of being symmetric for a Newton polygon, in order to draw an explicit analogy with the terminologies in the Siegel case.

Recall that for a point $x\in\mathscr{S},$ the $F$-isocrystal $M$ associated to $A_x$ admits a unique decomposition $M=\bigoplus_{\lambda\in\mathbb{Q}}M(\lambda)$ where $M(r)$ denotes the largest sub-isocrystal of slope $\lambda$. If $M$ is further equipped with an action by a finite dimensional $\mathbb{Q}$-algebra $B,$ then $M=\bigoplus_{v\in Spec{\mathcal{O}_F,v|p}}M_v(\lambda_v)$, where $F$ is the center of $B.$ The slopes $\lambda_v$ of $M_v$ are called the slopes of $M$ at $v.$ The multiplicity of the slope $\lambda_v$ is given by $$m_v(\lambda_v)=\frac{\dim_{L(k)}M_v(\lambda_v)}{[F_v:\mathbb{Q}_p][B:F]^{1/2}}.$$

Let $\nu$ denote a Newton polygon that comes from a $B$-linear polarized abelian variety. Then $\nu$ can be written as $\nu=\sum_{v\in \text{Spec}\mathcal{O}_F,v|p}\nu_v,$ where $$\nu_v=\sum_{i=1}^{N_v}m_{v,i}\lambda_{v,i}$$ for positive integers $n_v,m_{v,i}$. In the above notation, we use $\lambda_{v,i}$ to denote the distinct slopes of $\nu$ at a place $v$ of $F$ above $p$, and $m_{v,i}$ is the multiplicity of $\lambda_{v,i}.$

\begin{defn}\begin{enumerate}
    \item A Newton polygon $\nu$ that comes from a $B$-linear polarized abelian variety is \emph{$B$-balanced} if there exist positive integers $n,m$ such that $n_v=n$ for all primes $v$ of $F$ above $p$ and $m_{v,i}=m$ for all $v$ and $i.$
    \item Two Newton polygons are disjoint if they have no common slope at the same $v|p$ of $F.$ 
    \item A Newton polygon $\mu$ is \emph{$B$-symmetric} if it is the amalgamation of disjoint $B$-balanced Newton polygons.
\end{enumerate}
\end{defn}

In other words, for any $B$-symmetric Newton polygon $\mu$, there exists a positive integer $n$ and a multi-set $S$ such that for all $v|p$ of $F$, $n_v=n$ and $\{m_{v,i}\}_{i=1}^{n_v}=S$ as multi-sets for all $v|p$ of $F.$

It is clear from the definition that every $B$-balanced polygon is $B$-symmetric. Conversely, a $B$-symmetric polygon is naturally an amalgamation of uniquely determined disjoint $B$-balanced polygons.

We rephrase Zong's main theorem \cite[Theorem 5.1]{zong2008hypersymmetric} into a simplified version as follows:

\begin{thm}\label{hs}
A Newton stratum $\mathcal{N}$ contains a simple $B$-hypersymmetric point if and only if its Newton polygon is $B$-balanced.
\end{thm}

We remark that Theorem \ref{hs} follows from the proof but not the statement of \cite[Theorem 5.1.1]{zong2008hypersymmetric}. To make the present paper self-contained, we reproduce Zong's argument in the Appendix. The following corollary is an immediate consequence of Theorem \ref{hs}.

\begin{cor}
$\mathcal{N}$ contains a $B$-hypersymmetric point if and only if its Newton polygon is $B$-symmetric. 
\end{cor}

When $B=\mathbb{Q},$ the condition of being $B$-symmetric becomes empty, so Zong's result recovers \cite[Corollary 4.8]{chai2006hypersymmetric}.

\subsection{Hypersymmetricity over a subfield.} 
Recall from the notation of Section \ref{2.1} that $F_0$ is the maximal totally real field of the center $F$ of the simple $\mathbb{Q}$-algebra $B.$ For the proof of the main theorem, we need to understand when a $F$-hypersymmetric is also $F_0$-hypersymmetric. 

\begin{defn}\label{conditionstar}
Let $b\in B(G,\mu)$ and let $\zeta$ be the Newton polygon corresponding to $b.$ Write $\zeta=\oplus_{u|p}\zeta_u$ for primes $u|p$ of $F.$ We say $\zeta$ \emph{satisfies the condition (*)} if for any $u\neq u'$ of $F$ sitting above the same prime $v$ of $F_0$ above $p,$ the Newton polygons $\zeta_u$ and $\zeta_{u'}$ have no common slope.
\end{defn}

We prove the following consequence of Zong's criterion of hypersymmetry (\cite[Proposition 3.3.1]{zong2008hypersymmetric}).

\begin{prop}\label{inert} \begin{enumerate}
    \item Let $L/K$ be a finite extension of number fields such that every prime of $K$ above $p$ is inert in $L/K.$ Let $A$ be an $L$-hypersymmetric abelian variety defined over $\overline{\mathbb{F}_p},$ then $A$ is $K$-hypersymmetric.
    \item Let $K$ be a totally real field. Let $L/K$ be an imaginary quadratic extension where $p$ splits completely. Let $A$ be an $L$-hypersymmetric abelian variety defined over $\overline{\mathbb{F}_p}.$ Suppose the Newton polygon of $A$ satisfies the condition (*). Then $A$ is $K$-hypersymmetric.
\end{enumerate}

\end{prop}

\begin{proof} Let $A'$ be an abelian variety over some $\mathbb{F}_{p^a}$ such that $A'\otimes\overline{\mathbb{F}_p}\cong A.$ Let $\pi$ denote the $\mathbb{F}_{p^a}$-Frobenius endomorphism of $A'$ and let $\zeta$ denote the Newton polygon of $A.$ Then $\End_L(A)\otimes_{\mathbb{Z}}\mathbb{Q}=\End_L(A')\otimes_{\mathbb{Z}}\mathbb{Q},$ the center of which is identified with $L(\pi).$ Since $A$ is $L$-hypersymmetric, there exists a positive integer $n$ such that $\zeta$ has $n$ slopes at each $p|w$ of $L.$ By \cite[Proposition 3.3.1]{zong2008hypersymmetric}, $$L(\pi)\otimes_L L_w\cong \prod L_w,$$ with the number of factors equal to $n.$ 

(1) Suppose $L/K$ is inert. Let $u$ be the prime of $K$ below $w.$ Then $[L_w:K_u]=1,$ so $$\dim_KK(\pi)=\dim_LL(\pi)=n=n_u[L_w:K_u]=n_u,$$ where $n_u$ stands for the number of slopes of $\zeta$ at $u$. 

Therefore we have $K(\pi)\otimes_K K_v\cong \prod K_v$ with the number of factors being equal to $u.$ Since this holds for any $u|p$ of $K$, by \cite[Proposition 3.3.1]{zong2008hypersymmetric}, we conclude that $A$ is $K$-hypersymmetric.

(2) Suppose $K$ is totally real and $L/K$ is imaginary quadratic. Condition (*) implies $n_u=2n$ for any $u|p$ of $K.$ Since every prime above $p$ splits in $L/K$, we have $\dim_K K(\pi)=2\dim_L L(\pi).$ Again, this allows us to conclude by \cite[Proposition 3.3.1]{zong2008hypersymmetric} that $A$ is $K$-hypersymmetric.
\end{proof}

\begin{rmk}
In the case of $F/F_0,$ where $F_0$ is totally real and $F$ is a quadratic imaginary extension of $F_0,$ if the conditions in Proposition \ref{inert} are not satisfied, an $F$-balanced Newton polygon may not be $F_0$-balanced. Consider the following two examples:\begin{enumerate}
    \item Suppose $[F_0:\mathbb{Q}]=2$, $p=v_1v_2$ in $F_0$ and $v_1=u_1u_1', v_2=u_2u_2'$ in $F.$ Then the slope data 
    $$\begin{cases}
    \lambda, 1-\lambda &\text{at }u_1\\
    \lambda, 1-\lambda &\text{at }u_1'\\
    \mu, \nu &\text{at }u_2\\
    1-\mu, 1-\nu &\text{at }u_2'\\
    \end{cases}$$ where $\mu\neq \nu$ gives a $F$-balanced Newton polygon, but its restriction to $F_0$ given by 
    $$\begin{cases}
    2(\lambda),2(1-\lambda) &\text{at }v_1\\
    \mu, \nu, 1-\mu, 1-\nu &\text{at }v_2\\
    \end{cases}$$ is not $F_0$-balanced.
    \item Suppose $[F_0:\mathbb{Q}]=4$, $p=v_1v_2$ in $F_0,$ $v_1=u_1u_1'$ and $v_2$ is inert in $F/F_0.$ Then the slope data 
    $$\begin{cases}
    0 &\text{at }u_1\\
    1 &\text{at }u_1'\\
    1/2 &\text{at }v_2
    \end{cases}$$ where $\mu\neq \nu$ gives a $F$-balanced Newton polygon, but its restriction to $F_0$ given by 
    $$\begin{cases}
    0, 1 &\text{at }v_1\\
    2(1/2) &\text{at }v_2\\
    \end{cases}$$ is not $F_0$-balanced. 
\end{enumerate}
\end{rmk}

\subsection{Hypersymmetric points in the $\mu$-ordinary stratum in PEL type A}\label{mu} 

For this section, we restrict our attention to Shimura varieties of PEL type A. Namely, we assume $F$ is a CM field. Write $d=[F:\mathbb{Q}].$ We study necessary conditions on the $\mu$-ordinary Newton polygon to guarantee the existence of a $B$-hypersymmetric point.

Moonen explicitly computes the $\mu$-ordinary polygon in \cite{moonen2004serre} in terms of the multiplication type. Let $\mathcal{T}$ denote the set of complex embeddings of $F$ and $\mathfrak{O}$ denote the set of $\sigma$-orbits of elements in $\mathcal{T}.$ There is a bijection between $\mathfrak{O}$ and the set of primes of $F$ above $p.$ Let $\mathfrak{f}$ denote the multiplication type as defined in \cite[Section 0.4]{moonen2004serre}. For each $\sigma$-orbit $\mathfrak{o}$ of complex embeddings of $F,$ the corresponding $\mu$-ordinary Newton polygon has slopes (see \cite[Section 1.2.5]{moonen2004serre} and \cite[Definition 2.6.2]{eischen2017p}): $$a_j^{\mathfrak{0}}=\frac{1}{\#\mathfrak{o}}\#\{\tau\in\mathfrak{o}|\mathfrak{f}(\tau)>d-j\}\text{ for }j=1,\cdots,d.$$

For any $\lambda$ that occurs as a slope, the multiplicity of $\lambda$ is given by $$m_{\lambda}=\#\{j\in\{1,\cdots,d|a_j^{\mathfrak{o}}=\lambda\}.$$ Moonen's result makes it convenient to check the existence of $B$-hypersymmetric points. In particular, Example \ref{EmptyExample} demonstrates that not every $\mu$-ordinary stratum contains a $B$-hypersymmetric point.

%example incorrect
\begin{exmp}Suppose $[F:\mathbb{Q}]=4$. If $\mathfrak{o}_1=\{\tau_1,\tau_2\},\mathfrak{o}_2=\{\tau_1^*,\tau_2^*\}$with signature $(3,0),(1,4)$ respectively, then $p$ splits into $vv^*$ in $F$ and $\mu(v)=\mu(\mathfrak{o}_1)= (0)^1+(1/2)^3$ and $\mu(v^*)=\mu(\mathfrak{o}_2)=(1/2)^3+(1)^1.$ This $\mu$-ordinary stratum is $F$-symmetric and contains a hypersymmetric point.\end{exmp}

\begin{exmp}\label{EmptyExample}Suppose $[F:\mathbb{Q}]=4$. If $\mathfrak{o}_1=\{\tau_1,\tau_1^*\},\mathfrak{o}_2=\{\tau_2,\tau_2^*\}$ with signature $(3,1),(0,4)$ respectively, then  $p$ splits into two real places $v,v'$ in $F$ and $\mu(v)=\mu(\mathfrak{o}_1)=(0)^1+(1/2)^2+(1)^1$ and $\mu(v')=\mu(\mathfrak{o}_2)=(0)^2+(1)^2.$ This $\mu$-ordinary stratum contains no $B$-hypersymmetric point because the number of isotypic components above $v$ and $v'$ are different. \end{exmp}

Below we give a sufficient condition for the existence of $B$-hypersymmetric points in the $\mu$-ordinary stratum.

\begin{cor}\label{3.3} \begin{enumerate}
    \item Assume $p$ is inert in $F$. If every slope of the Newton polygon attached to $\mathcal{N}$ has the same multiplicity, then $\mathcal{N}$ contains a $B$-hypersymmetric point.
    \item Assume that the signature of $\mathscr{S}_{K^p}$ has no definite place, and that $p$ is a prime of constant degree in the extension $F/\mathbb{Q}$. Then the $\mu$-ordinary stratum contains a $B$-hypersymmetric point.
\end{enumerate}
\end{cor}

\begin{proof} (1) If $p$ is inert, there is only one $\sigma$-orbit. Thus the conditions of being $B$-balanced reduces to the condition that every slope has the same multiplicity, which is true by assumption.

(2) If $\mathfrak{f}$ has no definite place, $\mathfrak{f}$ only takes values in $[1,d-1]$. In this case, for any $\sigma$-orbit $\mathfrak{o},$ the number of values that $\mathfrak{f}$ takes in $[1,d-1]$ is precisely one less than the number of slopes of the $\mu$-ordinary polygon at the corresponding prime of $F$ above $p.$ Hence when the degree of $p$ is constant, the $\mu$-ordinary polygon has the same number of slopes at each prime of $F$ above $v$ and is therefore $B$-balanced.
\end{proof}

\section{Monodromy}\label{section 4}
An important ingredient in \cite{chai2005hecke} for proving the discrete part of the Hecke orbit conjecture is a monodromy result for Hecke invariant subvarieties (\cite[Proposition 4.1, Proposition 4.4]{chai2005monodromy}; c.f. \cite[Theorem 5.1]{chai2005hecke}). In this section, we present a generalization to all Shimura varieties of PEL type. Proposition \ref{thmk} is a more general version of the main theorem in \cite{kasprowitz2012monodromy}, which only holds if no simple factor of $\mathcal{D}$ is of type $D$. The key difference is that in the case of PEL type A and C, the derived group of $G$ is simply connected. In the case of PEL type $D$ we need to work with the simply connected cover of $G_{\text{der}}$ instead. The proofs we present in this section are closely related in that of loc. cit.

We first fix some notations. Let $\mathscr{S}_{K^p}$ and $G$ be as given in Section 2.1. Let $G':=G_{\text{der}}^{\text{sc}}$ denote the simply connected cover of the derived group of $G.$ 

Let $\Sigma$ be the finite set consisting of $p$ and the primes $p'$ such that some simple component of $G'$ is anisotropic over $\mathbb{Q}_{p'}.$ For a prime $l\neq p,$ let $Z\subseteq\mathscr{S}_{K^p}$ be a smooth locally-closed subvariety stable under all $l$-adic Hecke correspondences coming from $G'$. Let $Z^0$ be a connected component of $Z$ with generic point $z.$ We use $\overline{z}$ to denote a geometric point of $Z$ above $z.$ Let $A_{Z^0}$ denote the restriction to $Z^0$ of the universal abelian scheme over $\mathscr{S}_{K^p}$. Let $K_l$ be the image of $K^p$ under $K^p\hookrightarrow G(\mathbb{A}_f^p)\twoheadrightarrow G(\mathbb{Q}_l)$ and let $\rho_l:\pi_1(Z^0,\overline{z})\rightarrow K_l$ denote the $l$-adic monodromy representation attached to $A_{Z^0}$. Let $M=\rho_l(\pi_1(Z^0,\overline{z}))$ be the image of $\rho.$

As $N$ varies over all subgroups of $K^p$ for which $K^p/N$ is trivial away from $l$, we obtain the following pro-\'etale coverings:
$$(\mathscr{S}_N)_{N}\rightarrow \mathscr{S}_{K^p},$$ $$Y:=(\mathscr{S}_N\times_{ \mathscr{S}_{K^p}}Z^0)_N\rightarrow Z^0,$$ and $$\widetilde{Z}:=(\mathscr{S}_N\times_{\mathscr{S}_{K^p}}Z)_N\rightarrow Z,$$ where the first two admit $K_l$ action via $l$-adic Hecke correspondences, and the third one admits a natural $G'(\mathbb{Q}_l)$ action. Observe that $Aut_{\mathscr{S}_{K^p}}((\mathscr{S}_N)_N)=K^p$ by construction. Hence $\pi_1(Z^0,\overline{z})$ acts on $Y$ via the composition of morphisms $\pi_1(Z^0,\overline{z})\rightarrow \pi_1(\mathscr{S}^0_{K^p},\overline{z})\rightarrow Aut_{\mathscr{S}_{K^p}}((\mathscr{S}_N)_N)=K^p$, where $\mathscr{S}^0_{K^p}$ stands for the connected component of $\mathscr{S}_{K^p}$ containing $Z^0.$ Let $\widetilde{z}\in Y$ be a geometric point above $z$ and write $Y^0$ for the connected component of $Y$ passing through $\widetilde{z}.$ 

\begin{lem}\label{homeo}\begin{enumerate}
    \item There is a homeomorphism $\pi_0(Y)\cong K_l/M.$
    \item There is a homeomorphism $\pi_0(\widetilde{Z})\cong G'(\mathbb{Q}_l)/\Stab_{G'}(Y^0).$
\end{enumerate}
\end{lem}

\begin{proof}
The arguments in \cite[Section 2.6, 2.7 and Lemma 2.8]{chai2005monodromy} also work in the present situation. We remark that this is where the assumption of the transitivity of Hecke correspondences is used.
\end{proof}

\begin{lem}\label{Mss} Suppose $M$ is infinite. Then $M$ contains an open subgroup of $K_l.$
\end{lem}
\begin{proof} Let $H$ denote the neutral component of the Zariski closure of $M$ in $G.$ Let $\mathfrak{m}$ denote the Lie algebra of $M$ as an $l$-adic Lie group. By \cite[Corollary 7.9]{borel2012linear}, $\mathfrak{m}$ coincides with the Lie algebra of the Zariski closure $\overline{M}$, so $\mathfrak{m}$ contains the commutator subgroup of the Lie algebra $H$. Thus, if we can show $H=G'$, then $M$ contains an open subgroup of $G'$, which implies $M$ contains an open subgroup of $K_l.$ We do so by investigating the normalizer of $H$ in $G'$ and show that $H$ is in fact normal in $G'$.

By the same argument as in \cite[Proposition 4.1]{chai2005monodromy}, $H$ is semisimple, so $H\subseteq G'.$ Let $\mathbf{N}$ be the normalizer of $H$ in $G'$ and $\mathbf{N}^0$ its neutral component. The proof of \cite[Lemma 3.3]{chai2005monodromy} shows that $\mathbf{N}^0$ is reductive. 

Now we show $\mathbf{N}$ contains a nontrivial normal connected subgroup of $G'_{\mathbb{Q}_l}.$ There is a natural inclusion $\Stab_{G'}(Y)\subseteq \mathbf{N}(\mathbb{Q}_l),$ which gives rise to a continuous surjection from $G'(\mathbb{Q}_l)/\Stab_{G'}(Y)$ to $G'(\mathbb{Q}_l)/\mathbf{N}(\mathbb{Q}_l).$ 
By Lemma \ref{homeo}, the set on the left is profinite and in particular compact, so the group on the right is also compact. 
Thus $G'(\mathbb{Q}_l)/\mathbf{N}^0(\mathbb{Q}_l)$ is compact. By \cite[Proposition 9.3]{borel1965groupes}, $\mathbf{N}^0$ contains a maximal connected solvable subgroup $A$ of $G'_{\mathbb{Q}_l}.$ 
By the assumption on $l$, $G'_{\mathbb{Q}_l}$ is isotropic, so \cite[Propositions 8.4, 8.5]{borel1965groupes} imply that the set of unipotent elements $A_u$ is the unipotent radical of a minimal parabolic subgroup of $G'_{\mathbb{Q}_l}$. Hence, $A_u$ is nontrivial, connected, and normal.

Therefore, we must have $\mathbf{N}^0=\mathbf{N}=G'_{\mathbb{Q}_l}$ and $H$ is a normal subgroup of $G'$ having infinite intersections with all simple components of $G'$. We conclude that $H=G',$ which completes the proof.
\end{proof}

\begin{prop}\label{kprop3.1}
Notations and conditions as in Proposition \ref{thmk}. Suppose $M$ is infinite. Then $Z$ is connected.
\end{prop}

\begin{proof}
By Lemma \ref{Mss}, $M$ contains an open subgroup of $K_l.$ Hence, $K_l/M$ is finite. By Lemma \ref{homeo}, $\pi_0(Y)$ is also finite. Since $Z$ is quasi-projective, it has finitely many connected components, $\pi_0(\widetilde{Z})$ is finite as well. Again by Lemma \ref{homeo}, this implies that $G'(\mathbb{Q}_l)/\Stab_{G'}(Y^0)$ is finite. 
 
If $G'(\mathbb{Q}_l)/\Stab_{G'}(Y^0)\neq\{1\},$ then $\Stab_{G'}(Y^0)$ would be a non-trivial subgroup of finite index. The Kneser-Tits conjecture for simple and simply connected $\mathbb{Q}_l$-isotropic groups implies that none of the simple components of $G'$ has non-trivial non-concentral normal subgroups (\cite[Theorem 7.1, 7.6]{platonov1992algebraic}), and hence, no non-trivial subgroup of finite index. This is a contradiction. We conclude $\pi_0(\widetilde{Z})=G'(\mathbb{Q}_l)/\Stab_{G'}(Y^0)=\{1\}.$ In particular, $Z$ is connected.
 \end{proof}

The following proposition, generalizing the main theorem of \cite{kasprowitz2012monodromy}, is the upshot of this section.

\begin{prop}\label{thmk}
Suppose that the prime-to-$\Sigma$ Hecke correspondences from $G'$ act transitively on the set of connected components of $Z.$ If $z$ is not in the basic stratum, then $Z$ is connected. 
\end{prop}

\begin{proof}
By Proposition \ref{kprop3.1}, it suffices to show that $M=\rho_l(\pi_1(Z^0,\overline{z}))$ is infinite for all $l\notin \Sigma$. Suppose towards a contradiction that $M$ is finite for some $l$. By \cite[Theorem 2.1]{oort1974subvarieties}, there exists a finite surjective base change $Z'\rightarrow Z^0$ such that $A_{Z^0}\times_{Z^0}Z'$ is isogenous to an isotrivial abelian scheme $\mathbf{A}$ defined over $\overline{\mathbb{F}_p}.$ By \cite[Proposition I.2.7]{faltings2013degeneration}, the isogeny $A_{Z^0}\times_{Z^0}Z'\rightarrow \mathbf{A}\times_{\overline{\mathbb{F}_p}}Z'$ over $Z'$ extends to an isogeny over $\overline{Z'}.$ Hence, $\overline{Z^0}$ lies in a single Newton stratum. We claim that $\overline{Z^0}$ contains a basic point, which contradicts the assumption that the generic point of $Z^0$ lies outside the basic stratum. 

We first show that $\overline{Z^0}$ is a proper scheme over $\overline{\mathbb{F}_p}$. Let $R$ be a discrete valuation ring over $\mathcal{O}_E\otimes_{\mathbb{Z}}\mathbb{Z}_{(p)}$ with fraction field $K$, and let $(A,\lambda,i,\overline{\eta})$ be a $K$-valued point of $\overline{Z^0}.$ Then the N\'{e}ron model of $A$ over $R$ is in fact an abelian scheme. It is straightforward to check that $\lambda,i$ and $\overline{\eta}$ give a PEL structure on the N\'{e}ron model of $A$. Hence, $(A,\lambda,i,\overline{\eta})$ extends to an $R$-valued point of $\overline{Z^0}$ and $\overline{Z^0}$ is proper by the valuative criterion of properness.

Now we follow closely the proof of \cite[Proposition 6]{chai1995every} to show that $\overline{Z^0}$ contains a basic point. By assumption, the generic point $z$ of $Z^0$ is not in the basic stratum, so $\overline{Z^0}$ is positive dimensional. Since each Ekedahl-Oort stratum is quasi-affine, $\overline{Z^0}$ cannot be contained in the generic Ekedahl-Oort stratum. Hence, it must intersect some smaller stratum $S^{\omega_1}.$ By definition, each Ekedahl-Oort stratum is closed under $l$-adic Hecke correspondences, so $\overline{Z^0}\cap S^{\omega_1}$ is closed under $l$-adic Hecke correspondences as well. If $\overline{Z^0}\cap S^{\omega_1}$ is not $0$-dimensional, then its closure meets some smaller stratum $S^{\omega_2}.$ We can repeat this argument, and eventually reach a stratum $S^{\omega}$ such that $\overline{Z^0}\cap S^{\omega}$ is non-empty, $0$-dimensional, and closed under $l$-adic Hecke correspondences. Thus, $\overline{Z^0}$ contains a point whose $l$-adic Hecke orbit is finite. By \cite[Proposition 4.8]{yu2005basic}, this point must be basic. This completes the contradiction and the proof of our proposition.
\end{proof}

We apply Proposition \ref{thmk} to Newton strata and central leaves to obtain the following corollary.
\begin{cor}\label{equiv}\begin{enumerate}
    \item Suppose $\mathcal{N}$ is a non-basic Newton stratum. Further assume $\mathcal{N}$ is smooth. Then the prime-to-$\Sigma$ Hecke correspondences from $G'$ act transitively on $\pi_0(\mathcal{N})$ if and only if $\mathcal{N}$ is irreducible.
    \item Let $C\subseteq \mathscr{S}_{K^p}$ be a central leaf not contained in the basic stratum. Then the prime-to-$\Sigma$ Hecke correspondences from  $G'$ act transitively on $\pi_0(C)$ if and only if $C$ is irreducible.
\end{enumerate}
\end{cor}

\begin{rmk}\label{smooth}
Shen and Zhang \cite[Proposition 6.2.7]{shen2017stratifications} proved that non-basic Newton strata on Shimura varieties of abelian type are smooth if the pair $(G, \mu)$ is fully Hodge-Newton decomposable. G{\"o}rtz, He and Nie classified such pairs in \cite[Theorem 3.5]{goertz2019fully}.
\end{rmk}

\section{The discrete part} In this section, we restrict to cases A and C. We prove the discrete part of the Hecke Orbit conjecture under the assumption that the Newton stratum in question is irreducible and contains a $B$-hypersymmetric point. Namely,
\begin{thm}\label{irred}
Suppose $\mathcal{N}$ is a Newton stratum. Further assume that $\mathcal{N}$ contains a $B$-hypersymmetric point $x_0$ in some irreducible component $\mathcal{N}^0.$ Then $H^p$ acts transitively on $\Pi_0(C(x)\cap \mathcal{N}^0)$ for any $x\in\mathcal{N}^0.$ Moreover, if $\mathcal{N}$ is not the basic stratum, then $C(x)\cap \mathcal{N}^0$ is irreducible.
\end{thm}

\begin{lem}\label{weakapp}
Let $(A,\lambda,i,\overline{\eta})$ be an abelian variety with PEL structure over $k.$ Let $I_B$ be the unitary group attached to $(\End_B^0(A),*)$ where $*$ denotes the Rosati involution, i.e. for every commutative algebra $R,$ $$I_B(R)=U(\End_B(A)\otimes R,*)=\{u|u\cdot u^*=1=u^*\cdot u\}.$$ Then $I_{B}$ satisfies weak approximation.
\end{lem}
%\hl{revise}
\begin{proof}
Let $I$ denote the unitary group attached to $(\End^0(A),*)$. Then $I$ is connected (see the proof of \cite[Lemma 4.5]{chai2011monodromy}). By \cite[\textsection 2.11]{humphreys2011conjugacy}, $I_B$ is connected. By \cite[Lemma 4.6]{chai2011monodromy}, $I$ is $\mathbb{Q}$-rational, so is $I_B$. Then \cite[Proposition 7.3]{platonov1992algebraic} implies $I_B$ satisfies weak approximation.
\end{proof}

Now we are ready to prove Theorem \ref{irred}. The key idea of the proof is as follows. Denote by $\mathcal{N}^0$ the irreducible component of $\mathcal{N}$ containing $x_0$. For a central leaf $C$ in $\mathcal{N}$ such that $C\cap \mathcal{N}^0$ is nonempty, write $C^0=C\cap \mathcal{N}^0.$ We first use the almost product structure on Newton strata to show that there exist mutually isogenous $B$-hypersymmetric points on each  irreducible component of $C^0$. We then show that these points are in the same prime-to-$p$ Hecke orbit. This proves the first statement. Then we check the conditions of Corollary \ref{equiv} are satisfied and use it to conclude $C^0$ is irreducible. 

\begin{proof} \textbf{Step 1.} Denote by $\{C^0_j\}_{j\in J}$ the set of irreducible component of $C^0.$ By the product structure of Newton polygon strata (see \cite[Theorem 5.3]{oort2004foliations} and \cite[\textsection 4]{mantovan2004certain}), for $N, m, n, d$ large enough, there is a finite surjective morphism $$\pi_N:Ig_{m,\mathbb{X}}\times\overline{\mathcal{M}}_{\mathbb{X}}^{n,d}\rightarrow \mathcal{N}^0$$ such that for some closed geometric point $t$ of $\overline{\mathcal{M}}_{\mathbb{X}}^{n,d},$ $\pi_N$ restricts to a finite surjective morphism $$q_m: Ig_{m,\mathbb{X}}\times\{t\}\rightarrow C^0.$$ 

For any fixed $j\in J,$ let $(s_j,t_j)\in Ig_{m,\mathbb{X}}\times\overline{\mathcal{M}}_{\mathbb{X}}$ be such that $(s_j,t_j)\in \pi_N^{-1}(x_0)$ and $s_j\in q_m^{-1}(C^0_j),$ then $y_j=q_m(s_j)$ is a point in $C^0_j$ related to $x_0$ by a quasi-isogeny $\phi.$ Thus we obtain a set of points $\{y_j\in C^0_j\}_j$ that are mutually isogenous. Note that $y_j$ are $B$-hypersymmetric because the property of being hypersymmetric by definition is preserved under isogenies.

\textbf{Step 2.} Now we show that the $y_j$'s in \textbf{Step 1} are related by prime-to-$p$ isogenies. For any $i,j\in J,$ let $A_i$ and $A_j$ denote the abelian varieties with additional structure corresponding to $y_i$ and $y_j$, respectively. By construction, there exists an isogeny $\phi:A_i\rightarrow A_j.$
Since $y_i$ and $y_j$ belong to the same central leaf $C$, there exists an isomorphism $\theta_p:A_i[p^{\infty}]\xrightarrow{\sim} A_j[p^{\infty}].$ Let $\psi_p:A_i[p^{\infty}]\rightarrow A_i[p^{\infty}]$ be given by the composition $\phi_p^{-1}\circ \theta_p,$ where $\phi_p$ denotes the isogeny of Barsotti-Tate groups induced by $\phi$, then we have $\psi_p\in U(\End_B(A_i[p^{\infty}]),*).$ Since $A_i$ is hypersymmetric, the latter group is isomorphic to $I_B\otimes \mathbb{Q}_p$. By weak approximation for algebraic groups, there exists $\psi\in I_B(\mathbb{A}_f^p)$ that induces the same isogeny as $\psi_p$ on the Barsotti-Tate groups. Hence, $\phi\circ \psi:A_i\rightarrow A_j$ is an isogeny that induces an isomorphism on the Barsotti-Tate groups. Thus, any $y_i$ and $y_j$ are in the same prime-to-$p$ Hecke orbit. Therefore, $H^p$ acts transitively on the set of irreducible components of $C^0.$

In particular, since the Hecke symmetries on the connected component of $\mathscr{S}_{K^p}$ comes from the adjoint group $G^{\text{ad}}$ of $G,$ and $G^{\text{der}}$ is a covering of $G^{\text{ad}}$  (see \cite[1.6.5]{moonen1998models} and \cite{deligne1971travaux}), we may take the isogenies in Step 1 to be from $G_{\text{der}}$. Moreover, using weak approximation, we may take these isogenies to be prime-to-$\Sigma$, where $\Sigma$ is as defined in the beginning of Section \ref{section 4}. We conclude by Corollary \ref{equiv}(2) that $C^0$ is irreducible.
\end{proof}

\section{The continuous part at $B$-hypersymmetric points}
In this section, we prove the discrete part of the Hecke orbit conjecture for points that are $B$-hypersymmetric. We do not need to assume working in an irreducible (component of a) Newton stratum. 

\begin{thm}\label{cts}
Let $\mathcal{N}$ be a Newton stratum. Suppose $\mathcal{N}$ contains a $B$-hypersymmetric $x_0$. Let $C=C(x_0)$ denote the central leaf containing $x_0.$ Let $H$ denote the Zariski closure of $H^p(x_0)$ inside $C.$ Then $\dim H=\dim C.$ 
\end{thm}

\begin{rmk}
We remark that $H$ is smooth. Indeed, by generic smoothness, each connected component of $H$ has a smooth open dense subscheme, but each connected component of $H$ is by definition the union of prime-to-$p$ Hecke translates of that smooth open dense subscheme and therefore is smooth. Furthermore, by Proposition \ref{thmk}, $H$ is connected. Therefore, $H$ lies inside a single connected component of $C$.
\end{rmk}

In the case of Siegel modular varieties, the continuous part (see \cite[Theorem 10.6]{chai2005hecke} uses the ``Hilbert trick'' (\cite[Proposition 9.2]{chai2005hecke}) to find a point in $\mathcal{N}$ isogenous to a product of abelian varieties with at most two slopes - such a point is called ``split", thereby reducing the proof to the case where $x_0$ has at most two slopes. The Hilbert trick does not hold for PEL type in general. We observe that it is not necessary to work with a split point. We work instead with the full cascade structure on the formal completion of a central leaf and reduce to a statement analogous to what appears in the proof in the case of split points.

\subsection{The cascade structure on central leaves.}

In \cite{moonen2004serre}, Moonen generalizes classical Serre-Tate theory to $\mu$-ordinary points on PEL type Shimura varieties. He proves that the local deformation space at a generic point on a PEL type Shimura variety is built up from Barsotti-Tate groups via a system of fibrations over the Witt ring of $k$. For points outside the generic Newton stratum, Chai develops an analogous theory in the case when one restricts to a central leaf (see \cite[Sections 4.3, 4.4]{chai2006hecke}). In this section, we give a brief overview of the theory following loc. cit.

Let $A\rightarrow C$ be the restriction to $C$ of the universal abelian variety, and let $X$ be its Barsotti-Tate group with action by $\mathcal{O}_B\otimes_{\mathbb{Z}}\mathbb{Z}_p$. Then $X$ admits a unique slope filtration $0=X_0\subseteq X_1\subseteq\cdots\subseteq X_r=A[p^{\infty}]$ such that each $Y_i = X_i/X_{i-1}$ is a non-trivial isoclinic Barsotti-Tate group, and that the slopes of the $Y_i$'s appear in descending order (see \cite[Proposition 4.1]{chai2006hecke}). For $1\le i<j\le r,$ we use $\mathfrak{DE}(i,j)$ to denote the deformation space of the filtered Barsotti-Tate group $0\subseteq Y_i\subseteq Y_i\times Y_j$ with action by $\mathcal{O}_B\otimes_{\mathbb{Z}}\mathbb{Z}_p$, and let $\mathfrak{Def}(i,j)$ denote the deformation space of the filtered Barsotti-Tate group $0\subseteq X_i/X_{i-1}\subseteq X_{i+1}/X_{i-1}\subseteq\cdots\subseteq X_j/X_{i-1}$ with action by $\mathcal{O}_B\otimes_{\mathbb{Z}}\mathbb{Z}_p$. By definition, we have $\mathfrak{Def}(i,i+1)=\mathfrak{DE}(i,i+1)$ for any $i.$

The central leaf $C$ is homogeneous in the sense that formal completions of $C$ at any two points are non-canonically isomorphic. Thus it suffices to study what happens at the point $x_0\in C.$ The formal completion $C^{/x_0}$ is contained in the deformation space of the above-mentioned slope filtration, which admits a $r$-cascade structure in the sense of \cite[Definition 2.2.1]{moonen2004serre}. Denote this $r$-cascade by $\mathfrak{MDE}(X)$.
For $1\le i<j\le r,$ the group constituents of $\mathfrak{MDE}(X)$ are given by $\mathfrak{DE}(i,j)$, and the $(i,j)$-truncations are given by $\mathfrak{Def}(i,j).$ The $r$-cascade structure can be expressed in the following commutative diagram:

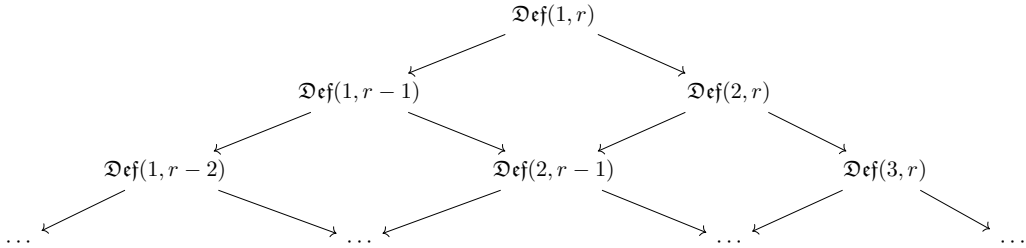
\adjustbox{scale=0.76,center}{%
    \begin{tikzcd}
& & & \mathfrak{Def}(1,r)\arrow{dl} \arrow{dr}  & & & \\
& & \mathfrak{Def}(1,r-1) \arrow{dl}\arrow{dr} & &  \mathfrak{Def}(2,r)\arrow{dl}\arrow{dr}  & & \\
& \mathfrak{Def}(1,r-2) \arrow{dl}\arrow{dr}& & \mathfrak{Def}(2,r-1)\arrow{dl}\arrow{dr}& & \mathfrak{Def}(3,r)\arrow{dl}\arrow{dr} &\\
\cdots&&\cdots&&\cdots&&\cdots\\
\end{tikzcd}
}

Here each $\mathfrak{Def}(i,j)$ is a bi-extension of $(\mathfrak{Def}(i,j-1), \mathfrak{Def}(i+1,j))$ by $\mathfrak{DE}(i,j)\times_{\text{Spec}(k)} \mathfrak{Def}(i+1,j-1)$.

For a smooth formal group $G$, we denote by $G_{\text{pdiv}}$ its maximal Barsotti-Tate subgroup. We write $\mathfrak{MDE}(X)_{\text{pdiv}}$ to mean the sub-cascade of $\mathfrak{MDE}(X)$ whose group constituents are $\mathfrak{DE}(i,j)_{\text{pdiv}}.$ Then $C^{/x_0}\subseteq \mathfrak{MDE}(X)$ is precisely the sub-cascade of $\mathfrak{MDE}(X)_{\text{pdiv}}$ fixed under the involution induced by the polarization of $x$ \cite[Section 4.4]{chai2006hecke}. We denote this sub-cascade by $\mathfrak{MDE}(X)_{\text{pdiv}}^{\lambda}$; its group constituents are $\mathfrak{DE}(i,j)_{\text{pdiv}}^{\lambda}$ when $i+j=r+1$ and $\mathfrak{DE}(i,j)_{\text{pdiv}}$ otherwise. This cascade structure on $C^{/x_0}$ allows us to reduce the proof of Theorem \ref{cts} to Proposition \ref{twoslope} below, which is an analogue of \cite[Theorem 7.3]{chai2005hecke} used in the Siegel case.

For the rest of this subsection, we study the group constituents of $\mathfrak{MDE}(X)_{\text{pdiv}}^{\lambda}$. Let $X,Y$ be isoclinic Barsotti-Tate groups over $k$ with $\mathcal{O}_B\otimes_{\mathbb{Z}}\mathbb{Z}_p$-action such that the slope of $X$ is smaller than that of $Y.$ Let $\mathfrak{DE}(X,Y)$ denote the deformation space of the filtered Barsotti-Tate group $0\subseteq X\subseteq X\times Y.$

\begin{notation}
If $G$ is a Barsotti-Tate group, write $M(G)$ for the Cartier module of $G.$ It can be equipped with the structure of a $V$-isocrystal, where $V$ denotes the dual operator to the Frobenius $F.$ For a polarization $\lambda,$ we write $G^{(\lambda)}$ to mean $G^{\lambda}$ when the induced action of $\lambda$ on $G$ is nontrivial and $G$ otherwise. We also recall that $W=W(k)$ denotes the Witt ring of $k$ and $L$ denotes the fraction field of $W.$
\end{notation}

\begin{prop}\label{twoslope}\begin{enumerate}
    \item There is a natural isomorphism of $V$-isocrystals $$M(\mathfrak{DE}(X,Y)_{\text{pdiv}})\otimes_{W}L\cong \Hom_{W}(M(X),M(Y))\otimes_{W}L.$$
    \item Let $\lambda$ be a principal quasi-polarization on $X\times Y.$ There is a natural isomorphism of $V$-isocrystals $$M(\mathfrak{DE}(X,Y)_{\text{pdiv}}^{\lambda})\otimes_{W}L\cong \Hom_{W}^{\lambda}(M(X),M(Y))\otimes_{W}L,$$ where by abuse of notation we use $\lambda$ to denote the involutions induced by $\lambda$ on the respective spaces.
\end{enumerate}\end{prop}
\begin{proof}
This clearly follows from \cite[Theorem 9.6]{chai2005canonical}, where the same statements are proved without assuming the presence of $\mathcal{O}_B\otimes_{\mathbb{Z}}\mathbb{Z}_p$-action.
\end{proof}

\subsection{Proof of the continuous part at $B$-hypersymmetric points.}

Notations as in Theorem \ref{cts}. The key idea in proving Theorem \ref{cts} is to study the action of the local stabilizer group at $x_0$ and show that the formal completion $H^{/x_0}\subseteq C^{/x_0}$ in fact coincides with $C^{/x_0}.$

\begin{defn}{\cite[Section 6.1]{chai2006hecke}} Let $x=[(A_x,\lambda_x,\iota_x,\eta_x)]\in\mathscr{S}(k)$ be a geometric point. Let $\mathcal{U}_x$ be the unitary group attached to the semisimple algebra with involution $(\End_{\mathcal{O}_B}(A_x)\otimes_{\mathbb{Z}}\mathbb{Q},*),$ where $*$ is the Rosati involution attached to $\lambda_x.$ We call $\mathcal{U}_x(\mathbb{Z}_p)$ the \emph{local stabilizer group at $x$}.
\end{defn}

By the assumption of Theorem \ref{cts}, $x_0$ is $B$-hypersymmetric, so $\mathcal{U}_{x_0}(\mathbb{Z}_p)$ coincides with $I_B(\mathbb{Z}_p)$ for the group $I_B$ defined in Lemma 5.2. By deformation theory, there is a natural action of $\mathcal{U}_{x_0}(\mathbb{Z}_p)$ on the formal completion $\mathscr{S}^{/x_0}$ and hence on its closed formal subschemes $\mathfrak{Def}(1,r)_{\text{pdiv}}^{(\lambda)}$ and $H^{/x_0}.$

Recall that the maps in the cascade structure of $\mathfrak{MDE}(X)$ are given by \cite[Proposition 2.1.9]{moonen2004serre} (see also \cite[2.3.6]{moonen2004serre}). These maps give rise to a group action of $\mathcal{U}_{x_0}(\mathbb{Z}_p)$ on each $\mathfrak{Def}(i,j)_{\text{pdiv}}$ in an inductive way. We describe this action in the next paragraph. 

Suppose that for the triple $\mathfrak{Def}(i,j-1)_{\text{pdiv}}\xleftarrow{p}\mathfrak{Def}(i,j)_{\text{pdiv}}\xrightarrow{r}\mathfrak{Def}(i+1,j)_{\text{pdiv}}$ we have an action of $\mathcal{U}_{x_0}(\mathbb{Z}_p)$ on $\mathfrak{Def}(i,j)_{\text{pdiv}}.$ Define an action of $\mathcal{U}_{x_0}(\mathbb{Z}_p)$ on $\mathfrak{Def}(i,j-1)_{\text{pdiv}}$ by $u\cdot \mathfrak{X}:=p(u\cdot\tilde{\mathfrak{X}})$ for any $u\in \mathcal{U}(\mathbb{Z}_p)$ and $\mathfrak{X}\in \mathfrak{Def}(i,j-1)_{\text{pdiv}},$ where $\tilde{\mathfrak{X}}$ is any pre-image of $\mathfrak{X}$ in $\mathfrak{Def}(i,j)_{\text{pdiv}}$ under $p.$ For any other choice of pre-image $\tilde{\mathfrak{X}}',$ by the biextension structure on $\mathfrak{Def}(i,j)_{\text{pdiv}}$, $\tilde{\mathfrak{X}}$ and $\tilde{\mathfrak{X}}'$ are in the same $\mathfrak{DE}(i+1,j-1)_{\text{pdiv}}$-orbit in $\mathfrak{Def}(i,j)_{\text{pdiv}}$. Hence, $u\cdot\tilde{\mathfrak{X}}$ and $u\cdot\tilde{\mathfrak{X}}'$ are also in the same $\mathfrak{DE}(i+1,j-1)_{\text{pdiv}}$-orbit, which implies that their images under $p$ coincide. Therefore this action of $\mathcal{U}_{x_0}(\mathbb{Z}_p)$ on $\mathfrak{Def}(i,j-1)_{\text{pdiv}}$ is well-defined. One defines the action on $\mathfrak{Def}(i+1,j)_{\text{pdiv}}$ in the obvious analogous way.

In the following, we study how $H^{/x_0}$ behaves with respect to the cascade structure of $C^{/x_0}\cong \mathfrak{MDE}(X)_{\text{pdiv}}^{\lambda}.$

For $1\le i \le r-1,$ let $H_{i,i+1}$ denote the image of $H^{/x_0}$ inside $\mathfrak{DE}(i,i+1)_{\text{pdiv}}^{(\lambda)}$ under the maps in the cascade structure of $C^{/x_0}.$ 

For $1\le i<j-1\le r$, we use the biextension structure of $\mathfrak{Def}(i,j)_{\text{pdiv}}^{(\lambda)}$ and define $H_{i,j}$ as the quotient $\mathfrak{Def}(i,j)^{(\lambda)}_{\text{pdiv}}/(H_{i,i+1}\times H_{j-1,j})$. By \cite[Section2]{mumford1969bi}, there is a non-canonical injection $\mathfrak{DE}(i,j)_{\text{pdiv}}^{(\lambda)}\times_{\text{Spec}(k)} \mathfrak{Def}(i+1,j-1)\hookrightarrow H_{i,j}$.

\begin{lem}\label{lempdiv}
For $1\le i\le r-1,$ $H_{i,i+1}$ is a formal Barsotti-Tate subgroup of $\mathfrak{DE}(i,i+1)_{\text{pdiv}}^{(\lambda)}$. 
\end{lem}

\begin{proof}By definition, $H$ is stable under all prime-to-$p$ Hecke correspondences. Then \cite[Proposition 6.1]{chai2006hecke} implies $H^{/x_0}$ is stable under the action of $\mathcal{U}_{x_0}(\mathbb{Z}_p);$ hence, so are the $H_{i,i+1}$'s. As explained in Remark 6.2, $H$ is smooth, so $H^{/x_0}$ is formally smooth and hence irreducible. The action of $\mathcal{U}_{x_0}(\mathbb{Z}_p)$ on $\mathfrak{DE}(i,i+1)_{\text{pdiv}}^{(\lambda)}$ gives a homomorphism from  $\mathcal{U}_{x_0}(\mathbb{Z}_p)\otimes_{\mathbb{Z}_p}\mathbb{Q}_p$ to the unitary group attached to $\text{End}(\mathfrak{DE}(i,i+1)_{\text{pdiv}})^{(\lambda)}\otimes_{\mathbb{Z}_p}\mathbb{Q}_p$. Applying \cite[Theorem 4.3]{chai2008rigidity} finishes the proof.
\end{proof}

\begin{lem}\label{lemequal}
For $1\le i \le r-1,$ $H_{i,i+1}=\mathfrak{DE}(i,i+1)_{\text{pdiv}}^{(\lambda)}.$
\end{lem}

\begin{proof}
The natural action of $\mathcal{U}_{x_0}(\mathbb{Z}_p)$ on $H_{i,i+1}$ restricts to an action on $M(Y_i)\otimes_W L$. Since $x_0$ is $B$-hypersymmetric, we have $\mathcal{U}_{x_0}(\mathbb{Z}_p)=I_B(\mathbb{Z}_p).$ Thus $\overline{\mathcal{U}_{x_0}(\mathbb{Z}_p)}$ acting on $M(Y_i)\otimes_W \overline{L}$ is isomorphic to the standard representation $Std$ of $GL_{\dim Y_i}.$ 

If the polarization $\lambda:=\lambda_{x_0}$ is nontrivial on $\mathfrak{DE}(i,i+1)_{\text{pdiv}}$, we obtain a duality pairing between $M(Y_i)$ and $M(Y_{i+1}).$ In this case, the action of $\overline{\mathcal{U}_{x_0}(\mathbb{Z}_p)}$ on $M(Y_{i+1})\otimes_W \overline{L}$ is isomorphic to the dual of $Std.$ Thus, the action of $\overline{\mathcal{U}_{x_0}(\mathbb{Z}_p)}$ on $\Hom_{W}^{\lambda}(M(Y_i),M(Y_{i+1}))\otimes_{W}\overline{L}$ is isomorphic to $\text{Sym}^2Std,$ which is irreducible. By Proposition \ref{twoslope}(2), the action of $\overline{\mathcal{U}_{x_0}(\mathbb{Z}_p)}$ on  $M(\mathfrak{DE}(i,i+1)_{\text{pdiv}}^{\lambda})\otimes_{W}\overline{L}$ is isomorphic to an irreducible representation.

On the other hand, if the polarization is trivial on $\mathfrak{DE}(i,i+1)_{\text{pdiv}}$, the action of $\overline{\mathcal{U}_{x_0}(\mathbb{Z}_p)}$ on $\Hom_{W}(M(Y_i),M(Y_{i+1}))\otimes_{W}\overline{L}$ is the tensor product representation of the standard representations of $GL_{\dim Y_i}$ and $GL_{\dim Y_{i+1}}$, which is irreducible. By Proposition \ref{twoslope}(1), the action of $\overline{\mathcal{U}_{x_0}(\mathbb{Z}_p)}$ on $M(\mathfrak{DE}(i,{i+1})_{\text{pdiv}})\otimes_{W}\overline{L}$ is again isomorphic to an irreducible representation.

By Lemma \ref{lempdiv}, it makes sense to consider the Cartier module $M(H_{i,{i+1}})$. We have $M(H_{i,{i+1}})\otimes_W \overline{L}$ as a non-trivial sub-representation of $M(\mathfrak{DE}(i,{i+1})_{\text{pdiv}})^{(\lambda)}\otimes_W\overline{L}$ which is irreducible, so we obtain the desired equalities.
\end{proof}

\begin{proof}[Proof of Theorem \ref{cts}]
We show inductively that Lemma \ref{lemequal} implies $H^{x_0}=C^{/x_0}.$

When $r=2,$ $C^{/x_0}=\mathfrak{DE}(1,2)_{\text{pdiv}}^{\lambda}$ and there is nothing to prove.

When $r=3,$ $C^{/x_0}=\mathfrak{Def}(1,3)_{\text{pdiv}}^{\lambda}$ is a biextension of $(\mathfrak{DE}(1,2)_{\text{pdiv}},\mathfrak{DE}(2,3)_{\text{pdiv}})$ by $\mathfrak{DE}(1,3)_{\text{pdiv}}^{\lambda}.$ The equalities in Lemma 6.5 induces an isomorphism realized via $H_{1,3}=\mathfrak{Def}(1,3)_{\text{pdiv}}^{\lambda}/(H_{1,2}\times H_{2,3})=\mathfrak{Def}(1,3)_{\text{pdiv}}^{\lambda}/(\mathfrak{DE}(1,2)_{\text{pdiv}}\times \mathfrak{DE}(2,3)_{\text{pdiv}})\cong \mathfrak{DE}(1,3)_{\text{pdiv}}^{\lambda}.$ Thus, the only candidate for the subscheme $H^{/x_0}$ of $\mathfrak{Def}(1,3)_{\text{pdiv}}^{\lambda}$ is $\mathfrak{Def}(1,3)_{\text{pdiv}}^{\lambda}$ itself.

When $r\ge 3,$ let $H^{i,j}$ denote the image of $H^{/x_0}$ in $\mathfrak{Def}(i,j)_{\text{pdiv}}^{\lambda}.$ By induction we have equalities $H^{i,j}=\mathfrak{Def}(i,j)_{\text{pdiv}}^{(\lambda)}$ for all $(i,j)\neq (1,r).$ In particular, $H^{1,r-1}=\mathfrak{Def}(1,r-1)_{\text{pdiv}}$ and $H^{2,r}=\mathfrak{Def}(2,r)_{\text{pdiv}}$ together imply $$H
_{1,r}=\mathfrak{Def}(1,r)_{\text{pdiv}}^{\lambda}/(\mathfrak{DE}(1,r-1)_{\text{pdiv}}\times \mathfrak{DE}(2,r)_{\text{pdiv}}),$$ so we obtain $H^{/x_0}=\mathfrak{Def}(1,r)_{\text{pdiv}}^{\lambda}=C^{/x_0}.$
\end{proof}

\section{Proof of the main theorem}\label{section7}
The main theorem of our paper is the following.

\begin{thm}\label{mainthm}
Let $\mathcal{D}=(B,\mathcal{O}_B,*,V,\langle\cdot,\cdot\rangle, h)$ be a Shimura datum of PEL type A or C, for which $p$ is an unramified prime of good reduction. Let $F$ be the center of $B$ and $F_0$ its maximal totally real subfield. Let $\mathscr{S}$ denote the reduction modulo $p$ of the Shimura variety associated to $\mathcal{D}$ of level $K^p\subseteq G(\mathbb{A}_f^p)$. Let $\mathcal{N}$ be a Newton stratum on $\mathscr{S}$. Assume \begin{enumerate}
    \item $\mathcal{N}$ contains a hypersymmetric point $x_0$,
    \item either (i) $p$ is totally split in $F/F_0$ and the Newton polygon of $\mathcal{N}$ satisfies the condition (*) in Definition \ref{conditionstar}; or (ii) $p$ is inert in $F/F_0.$
\end{enumerate}
Write $\mathcal{N}^0$ for the irreducible component of $\mathcal{N}$ containing $x_0$. Then $H^p(x)$ is dense in $C(x)\cap \mathcal{N}^0$ for every $x\in\mathcal{N}^0(k).$ Moreover, if $\mathcal{N}$ is not the basic stratum, then $C(x)\cap \mathcal{N}^0$ is irreducible.
\end{thm}

We observe that Theorem \ref{irred} implies that for any $x\in\mathcal{N}^0$, $C(x)\cap \mathcal{N}^0$ is irreducible and hence coincides with the irreducible component of $C(x)$ containing $x$. We denote this component by $C^0(x)$. Moreover, Theorem \ref{irred} and Theorem \ref{cts} together imply that the Hecke orbit conjecture holds for any $B$-hypersymmetric point. Then the key in deriving Theorem \ref{mainthm} from Theorem \ref{irred} and Theorem \ref{cts} lies in showing the existence of a $B$-hypersymmetric point in $\overline{H^p(x)}\cap C^0(x)$ for every $x\in\mathcal{N}^0$. 

\subsection{The case of $B=F$.}
We consider first the situation where $B=F$ is a totally real field. In this case, our PEL datum is of type C. The algebraic group $G$ is symplectic, and the Hilbert trick still applies. We use the Hilbert trick to embed a Hilbert modular variety where every central leaf in the Newton strata corresponding to $\mathcal{N}$ contains $B$-hypersymmetric point.

For the remainder of this section, we fix a point $x=[(A_x,\lambda_x,\iota_x,\eta_x)]\in\mathcal{N}(\overline{\mathbb{F}_p}).$ There is a decomposition $A_x\sim_{F\text{-isog}} A_1^{e_1}\times\cdots\times A_n^{e_n}$ into $F$-simple abelian varieties $A_i$ such that $A_i$ and $A_j$ are not $F$-isogenous whenever $i\neq j.$ For each $i,$ let $E_i\subseteq \End_{F}(A_i)\otimes_{\mathbb{Z}}\mathbb{Q}$ be the maximal totally real subalgebra fixed under the involution given by the polarization of $A_i$ induced by $\lambda_x,$ then $E_i/F$ is a totally real field extension of degree $\dim(A_i)/[F:\mathbb{Q}].$ Define $E:=\prod_{i=1}^nE_i,$ then $E$ is a subalgebra of $\End_F(A_x)\otimes_{\mathbb{Z}}\mathbb{Q}$ of dimension $\dim(A).$

This construction relies on the assumption that $x$ is defined over $\overline{\mathbb{F}_p},$ so the method in this section only applies when $k=\overline{\mathbb{F}_p}.$ However, by \cite[Theorem 2.2.3]{poonen2017rational}, to prove the Hecke orbit conjecture over any algebraically closed field $k$ of characteristic $p,$ it suffices to prove it over $\overline{\mathbb{F}_p}.$

\begin{lem}\label{fields}
Let $F$ be a totally real number field and $d$ a positive integer. Let $u_1,\cdots,u_n$ be the places of $F$ above $p.$ Suppose that for each $i$, there is a finite separable extension $K_i/F_{u_i}$ of degree $d.$ Then there exists a totally real field extension $L$ of $F$ of degree $d$ such that all the $u_i$'s are inert in $L/F$ and $L_{w_i}\cong K_i$ over $F_{u_i}$ for $w_i$ of $L$ above $u_i.$
\end{lem}
\begin{proof}
For a place $u_i$ of $F$ above $p,$ we have $K_i\cong F_{u_i}(\alpha_i)$, where $\alpha_i$ is a root of some irreducible separable monic polynomial $f_i'\in F_{u_i}[X]$ of degree $d.$ By Krasner's lemma, we can approximate $f_i'$ by some irreducible separable monic polynomial $f_i\in F[X]$ such that $F_{u_i}[X]/(f_i)\cong F_{u_i}[X]/(f_i')\cong K_i.$ Let $v_1,\cdots,v_m$ denote the archimedean places of $F$ and let $g_i=\prod_{j=1}^d (X-\beta_{ij})$ for distinct $\beta_{ij}\in F,$ so $F_{v_i}[X]/(g_i)\cong \mathbb{R}[X]/(g_i)\cong \mathbb{R}^d$ since $F$ is totally real. By weak approximation, there exists some monic $f\in F[X]$ that is $u_i$-close to $f_i$ for each $i$ and $v_i$-close to $g_i$ for each $i.$ In particular, $f=f_i$ for some $i,$ so $f$ is irreducible in $F_{u_i}[X]$ and hence irreducible in $F.$

Let $L=F[X]/(f),$ then $L/F$ is separable of degree $d$. Moreover, for each $u_i$, we have $$\prod_{w|u_i}L_{w}\cong F_{u_i}\otimes L\cong F_{u_i}[X]/(f)\cong K_i.$$ This implies that there is a unique place $w_i$ of $E$ above $u_i$ and $L_{w_i}\cong K_i.$ Similarly, for each archimedean place $v_i$, we have $\prod_{w|v_i}L_w\cong F_{v_i}\otimes L\cong \mathbb{R}^d.$ Thus, $L$ is totally real.
\end{proof}

\begin{lem}\label{Eprime}
The closure of the prime-to-$p$ Hecke orbit of $x$ in $\mathscr{S}$ contains a supersingular point $z=[(A_z,\lambda_z,\iota_z,\eta_z)].$ Moreover, there exists a product of totally real fields $L=\prod_{i,j}L_{i,j}$ and an injective ring homomorphism $\alpha:L\rightarrow \End_F(A_z)\otimes_{\mathbb{Z}}\mathbb{Q}$ such that \begin{itemize}
    \item for every $i,j$, $F\subseteq L_{i,j}$ and $L_{i,j}/F$ is inert at every prime of $F$ above $p$;
    \item $\dim_{\mathbb{Q}}L=\dim A_z;$ and
    \item the Rosati involution induced by $\lambda_z$ acts trivially on the image $\alpha(L)$.
\end{itemize}
\end{lem}

\begin{proof}
First of all, the same argument as in the last paragraph of the proof of Proposition \ref{thmk} shows $\overline{H^p(x)}$ contains a basic point $z=[(A_z,\lambda_z,\iota_z,\eta_z)].$ Since $B=F$ is a totally real number field, $z$ is supersingular. Indeed, let $d=[F:\mathbb{Q}],$ then $F=\mathbb{Q}(x)$ where $x$ is the root of some degree $d$ monic irreducible polynomial $f\in\mathbb{Q}.$ Then endomorphism ring of a $d$-dimensional supersingular abelian variety is isomorphic to $\text{Mat}_d(D_{p,\infty}),$ where $D_{p,\infty}$ denotes the quaternion algebra over $\mathbb{Q}$ ramified exactly at $p$ and $\infty,$ and $\text{Mat}_d(D_{p,\infty})$ clearly contains the companion matrix of $f$.

By assumption, $\mathcal{N}$ admits an $F$-hypersymmetric point, so its Newton polygon $\zeta$ is $F$-symmetric, i.e. $\zeta$ is the amalgamation of disjoint $F$-balanced Newton polygons $\zeta_1,\cdots,\zeta_a$ for some $a.$ By \cite[Definition 4.4.1]{zong2008hypersymmetric}, for any $j,$ $\zeta_j$ either has no slope $1/2$, or only has slope $1/2.$ Hence, there exist positive integers $m_j, h_j$ such that at every prime $u$ of $F$ above $p,$ $\zeta_j$ has exactly $h_j$ symmetric components with at most $2$ slopes, and each component has multiplicity $m_j.$ 

Write $\mathbb{X}=A_x[p^{\infty}]$, then the above decomposition of $\zeta$ gives a decomposition $\mathbb{X}=\bigoplus_{j=1}^a\mathbb{X}_j,$ where $\mathbb{X}_j$ is the Barsotti-Tate group corresponding to $\zeta_j.$ Further, for each fixed $j,$ the numerical properties of $\zeta_j$ mentioned above gives a decomposition $$\mathbb{X}_j=\bigoplus_{u|p}\bigoplus_{i=1}^{h_j}\mathbb{X}^u_{j,i}$$ into Barsotti-Tate groups of at most $2$ slopes.

Hence, $$\End_F(\mathbb{X})\otimes_{\mathbb{Z}_p}\mathbb{Q}_p\cong \prod_{u|p}\prod_{j=1}^a\prod_{i=1}^{h_j}\End_{F_u}(\mathbb{X}_{j,i}^u).$$ Its subalgebra $E\otimes_{\mathbb{Q}} \mathbb{Q}_p$ similarly admits a decomposition $$E\otimes_{\mathbb{Q}}\mathbb{Q}_p=\prod_{u|p}\prod_{j=1}^a\prod_{i=1}^{h_j}E_{j,i}^u$$ into local fields. Note that this has to coincide with the decomposition $E\otimes\mathbb{Q}_p\cong \prod_{i=1}^nE_i\otimes\mathbb{Q}_p=\prod_{i=1}^n\prod_{w|p}E_w.$

Now we regroup the local data $\{E_{j,i}^u/F_u\}_{j,i,u}$ to construct totally real extension of $F.$ Notice that for a fixed $j,$ the numerical conditions on $\zeta_j$ implies that for any fixed $i,$\begin{itemize}
    \item $\#\{E_{j,i}^u\}_{u|p}=h_j,$ and
    \item $[E_{j,i}^u:F_u]=[E_{j,i}^{u'}:F_{u'}]=\dim(\mathbb{X}_{j,i}^u)/[F:\mathbb{Q}]$ for any primes $u,u'$ of $F$ above $p.$
\end{itemize}

By Lemma \ref{fields}, for a fixed pair $(j,i)$, the data $\{E_{j,i}^u\}_{u|p}$ gives rise to a totally real field extension $L_{i,j}/F$ with $[L_{i,j}:\mathbb{Q}]=\dim(\mathbb{X}_{j,i})$ such that all primes of $F$ above $p$ are inert and $\{(L_{i,j})_{w}\}_{w|p}=\{E_{i,j}^u\}_{u|p}$ as multi-sets. 

Taking $L=\prod_{i,j}L_{i,j},$ it is easy to see that $\dim_{\mathbb{Q}}L'=\dim A_z$.

In general, $\mathcal{O}_E\cap \End_{F}(A_z)\subseteq\mathcal{O}_E$ is of finite index. Following the argument in \cite[Theorem 11.3]{chai2005hecke}, up to an isogeny correspondence, we may assume $\End_F(A_z)$ contains $\mathcal{O}_E.$ Similarly, we may and do assume $\mathcal{O}_L\subseteq\End_F(A_z).$ By construction, we then have $\mathcal{O}_E\otimes_{\mathbb{Z}} \mathbb{Z}_p\cong \mathcal{O}
_L\otimes_{\mathbb{Z}}\mathbb{Z}_p$ as maximal orders of $\End_F(A_z)\otimes_{\mathbb{Z}}\mathbb{Z}_p,$ so the Noether-Skolem theorem implies that $E=\gamma L\gamma^{-1}$ for some $\gamma$ in the local stabilizer gorup of $z.$ Then $\alpha:=\text{Ad}(\gamma)$ satisfies the property in the statement of this lemma.
\end{proof}

\begin{prop}\label{B=F}
Theorem \ref{mainthm} holds when $B=F$ is a totally real field.
\end{prop}

\begin{proof}
If $x$ is (isogenous to) a $F$-hypersymmetric point, the desired statement follows immediately from Theorems \ref{irred} and \ref{cts}, so we assume $x$ is not $F$-hypersymmetric. 

We use an idea analogous to that of \cite[Section 10]{chai2005hecke} and show that $\overline{H^p(x)}\cap C^0$ contains an $F$-hypersymmetric point $t$. Then we have $\overline{H^p(t)}\cap C^0\subseteq \overline{H^p(x)}\cap C^0$, but $\overline{H^p(t)}\cap C^0=C^0$ since $t$ is $F$-hypersymmetric, and we are done.

Let $E$ be as given in the beginning of this section, and let $\mathcal{M}_E$ denote the Hilbert modular variety attached to $E.$ In general, $E\cap \End_{\overline{\mathbb{F}_p}}(A_x)$ is an order of the ring $\mathcal{O}_E.$ However, up to an isogeny correspondence, we may assume $\mathcal{O}_E\subseteq \End_{\overline{\mathbb{F}_p}}(A_x).$ Then there exists a finite morphism $\mathcal{M}_E\rightarrow\mathscr{S}$ passing through $x$, compatible with the prime-to-$p$ Hecke correspondence on $\mathcal{M}_E$ and $\mathscr{S}$ and such that for each geometric point of $\mathcal{M}_E,$ the map induced by $f$ on the strict henselizations is a closed embedding (for details, see \cite[Proposition 9.2]{chai2005hecke} and the proof of \cite[Theorem 11.3]{chai2005hecke}). 

Let $L=\prod_i{L_i}$ and $\gamma$ be as given by Lemma \ref{Eprime}. Again, up to an isogeny correspondence, we may assume $\mathcal{O}_{L}\subseteq \End_{\overline{\mathbb{F}_p}}(A_z)$. Similarly, there is a finite natural morphism $g:\mathcal{M}_{L}\rightarrow\mathscr{S}$ passing through $z,$ such that $g$ is compatible with the Hecke correspondences on either side and at every geometric point of $\mathcal{M}_{L}$ induces a closed embedding on the strict henzelizations. Moreover, writing $H^p_E(x)$ for the prime-to-$p$ Hecke orbit of $x$ in $\mathcal{M}_E,$ we have $\gamma^{/z}(\overline{H^p_E(x)})\subseteq g^{/z}(\mathcal{M}_{L}^{/z})$ and $\gamma^{/z}(\overline{H^p_E(x)})\subseteq \gamma^{/z}(\overline{H^p(x)}^{/z}).$ By \cite[Proposition 6.1]{chai2006hecke}, $\overline{H^p(x)}^{/z}$ is stable under the action of the local stabilizer group of $z,$ so $\gamma^{/z}(\overline{H^p(x)}^{/z})=\overline{H^p(x)}^{/z}$ and we obtain $\gamma^{/z}(\overline{H^p_E(x)}^{/z})\subseteq g^{/z}(\mathcal{M}_{L}^{/z})\cap \overline{H^p(x)}^{/z}.$ Therefore the fiber product $\mathcal{M}_{L}\times_{\mathscr{S}_{F}}(\overline{H^p(x)}\cap C^0(x))$ is nonempty. 

Now let $\tilde{y}$ be an $\overline{\mathbb{F}_p}$-point of $\mathcal{M}_{L}\times_{\mathscr{S}_F}(\overline{H^p(x)}\cap C^0)$ with image $\overline{y}\in\mathcal{M}_{L}(\overline{\mathbb{F}_p})$ and image $y\in (\overline{H^p(x)}\cap C^0)(\overline{\mathbb{F}_p}).$ By definition, we have $g(\overline{y})=y.$ Moreover, $\overline{H^p(x)}\cap C^0$ contains the image under $g$ of the smallest Hecke-invariant subvariety of $\mathcal{M}_{L}$ passing through $\overline{y},$ which by the Hecke orbit conjecture for Hilbert modular varieties \cite[Theorem 4.5]{chai2005hecke} is precisely the central leaf $C_{L}(\overline{y})$.
Therefore, if we can show $C_{L}(\overline{y})$ contains a point isogenous to a product of $L_i$-hypersymmetric points, then Proposition \ref{inert} implies that $C_{L}(\overline{y})$ contains a $F$-hypersymmetric point, and so does $\overline{H^p(x)}\cap C^0.$

To see that $C_L(\overline{y})$ contains a point isogenous to a product of $L_i$-hypersymmetric points, consider the canonical decomposition $\mathcal{M}_L\cong\prod\mathcal{M}_{L_i}$ and the corresponding Newton stratum decomposition $\mathcal{N}_L\cong\prod \mathcal{N}_{L_i}$. By the construction of $L_i,$ the Newton polygon of $\mathcal{N}_{L_i}$ either has exactly two slopes at every prime of $L_i$ above $p,$ or only has slope $1/2.$ In either case, $\mathcal{N}_{L_i}$ admits an $L_i$-hypersymmetric point. By the existence of finite correspondences between central leaves on $\mathcal{N}_{L_i}$ (see \cite[Section 1.3]{yustratifying}), each central leaf of $\mathcal{N}_{L_i}$ contains a ${L_i}$-hypersymmetric point. This completes our proof. 
\end{proof}

\subsection{Proof of the general case.} In this section, we complete the proof of Theorem \ref{mainthm}.

\begin{proof}[Proof of Theorem \ref{mainthm}] Again, we only need to prove the statement when $x$ is not $B$-hypersymmetric. 

\textbf{Case 1.} When $B$ is totally real, see Proposition \ref{B=F}.

\textbf{Case 2.} Suppose $B=F$ is a CM field. 

Let $\mathcal{D}'$ be the Shimura datum given by replacing $B$ by its maximal totally real subfield $F_0$ in the definition of $\mathcal{D}.$ Let $\mathscr{S}'$ denote the Shimura variety arising from $\mathcal{D}'$. Then there is an embedding $\mathscr{S}\rightarrow \mathscr{S}'$ given by sending any $[(A,\lambda,\iota,\eta)]$ to $[(A,\lambda,\iota|_{F_0},\eta)].$ Let $x=[(A,\lambda,\iota,\eta)]\in \mathcal{N}^0$ be any point, and we denote its image in $\mathscr{S}'$ also by $x$. By assumption, $\mathcal{N}^0$ contains a $F$-hypersymmetric point $x_0$. Proposition \ref{inert} implies $x_0$ is also $F_0$-hypersymmetric. Write $H^{p\prime}(x)$ for the prime-to-$p$ Hecke orbit of $x$ in $\mathscr{S}$ and $C^{0\prime}$ for the irreducible component of the central leaf in $\mathscr{S}'$ passing through $x$. Then by Proposition \ref{B=F}, $\overline{H^{p\prime}(x)}\cap C^{0\prime}=C^{0\prime}.$

Now we show that $H^{p\prime}(x)\cap C^0=H^p(x).$ Let $x'=[(A',\lambda',\iota',\eta')]\in H^{p\prime}(x),$ then there is a prime-to-$p$ $F_0$-isogeny $f$ between $x$ and $x'.$ By definition, $f\circ \iota|_{F_0}=\iota'|_{F_0}\circ f$. Since $F/F_0$ is a quadratic imaginary extension, $f$ extends to an $F$-isogeny between $x$ and $x'.$

Finally, we have $\overline{H^p(x)}\cap C^0=\overline{H^{p\prime}(x)\cap C^0}\cap C^0=\overline{H^{p\prime}(x)}\cap C^0=\overline{H^{p\prime}(x)} \cap C^{0\prime}\cap C^0=C^{0\prime}\cap C^0=C^0.$

\textbf{Case 3.} Now we are ready to show the statement for a general $B$.

Let $\mathcal{D}'$ be the Shimura datum given by replacing $B$ by its center $F$ in the definition of $\mathcal{D}.$ Let $\mathscr{S}'$ denote the Shimura variety arising from $\mathcal{D}'$. Then there is an embedding $\mathscr{S}\rightarrow \mathscr{S}'$ given by sending $[(A,\lambda,\iota,\eta)]$ to $[(A,\lambda,\iota|_{F},\eta)].$ Let $x\in \mathcal{N}^0$ be any point, and we denote its image in $\mathscr{S}'$ also by $x$. By assumption, $\mathcal{N}^0$ contains a $B$-hypersymmetric point $x_0$. By \cite[Proposition 3.3.1]{zong2008hypersymmetric}, $x_0$ is also $F$-hypersymmetric. Write $H^{p\prime}(x)$ for the prime-to-$p$ Hecke orbit of $x$ in $\mathscr{S}$ and $C^{0\prime}$ for the irreducible component of the central leaf in $\mathscr{S}'$ passing through $x$. Then by the previous two cases, $\overline{H^{p\prime}(x)}\cap C^{0\prime}=C^{0\prime}.$

Now we show that $H^{p\prime}(x)\cap C^0=H^p(x).$ Let $x'=[(A',\lambda',\iota',\eta')]$ be a closed geometric point of $\mathscr{S}$ such that there is an prime-to-$p$ $F_0$-isogeny $f$ from $x=[(A,\lambda,\iota,\eta)]$ to $x'$. Then $f$ induces a morphism $f:\End(A)\rightarrow \End(A')$ such that $f\circ \iota|_{F}=\iota'|_{F}\circ f$ on $F.$ By the Skolem-Noether Theorem, $f\circ \iota|_{F}=\iota'|_{F}\circ f$ extends to an inner automorphism $\varphi:B\rightarrow B.$ Hence, $f$ extends to a $B$-isogeny between $x$ and $x'.$

By an argument analogous to the last part of Case 2, we conclude $\overline{H^p(x)}\cap C^0=C^0.$
\end{proof}

\subsection{Special cases of the main theorem.}
Based on the discussions in section \ref{section 3}, we prove the following corollaries of the main theorem. 

\begin{cor}\label{cormu}\begin{enumerate}
    \item Suppose $p$ is inert in $F$. If every slope of the Newton polygon attached to $\mathcal{N}$ has the same multiplicity, then the Hecke orbit conjecture holds for  any irreducible component of $\mathcal{N}$ containing a $B$-hypersymmetric point.
    \item Suppose the center of $B$ is a CM field. Assume that the signature of $\mathscr{S}$ has no definite place, and that $p$ is a prime of constant degree in the extension $F /\mathbb{Q}$. Further assume assumption 2 in Theorem \ref{mainthm} is satisfied. Then the Hecke orbit conjecture holds for every irreducible component of the $\mu$-ordinary stratum.
\end{enumerate}
\end{cor}

\begin{proof}
By Corollary \ref{3.3}, in either of the two cases, the Newton stratum $\mathcal{N}$ contains a $B$-hypersymmetric point. In the second case, $G(\mathbb{A}_f^p)$ acts transitively on $\Pi_0(\mathscr{S})$, so if the $\mu$-ordinary stratum contains a $B$-hypersymmetric point, then it contains a $B$-hypersymmetric point in each of its irreducible components. Moreover, the assumptions satisfy the conditions of Proposition \ref{inert}. Hence, we may apply Theorem \ref{mainthm} and to derive the desired results.
\end{proof}

\begin{cor}\label{corachter}
Let $L$ be a quadratic imaginary field inert at the rational prime $p$. The Hecke Orbit conjecture holds for the moduli space of principally polarized abelian varieties of dimension $n\ge 3$ equipped with an action by $\mathcal{O}_L$ of signature $(1, n-1).$
\end{cor}

\begin{proof}
Since $p$ is inert, a Newton stratum contains a $L$-hypersymmetric point if its Newton polygon is symmetric. By \cite[Seciton 3.1]{bultel2006congruence}, any admissible Newton polygon in this case is given by $N(r)+(1/2)^{n-2r}$ for some integer $0\le r\le n/2$, where \begin{equation*}
    N(r)=\begin{cases}\emptyset,&\text{if }r=0,\\
    (\frac{1}{2}-\frac{1}{2r})+(\frac{1}{2}+\frac{1}{2r}), &\text{if }r>0\text{ is even},\\
    (\frac{1}{2}-\frac{1}{2r})^2+(\frac{1}{2}+\frac{1}{2r})^2, &\text{if }r\text{ is odd}.
    \end{cases}
\end{equation*}
From this description, it follows by easy computation that for each pair $(n,r),0\le r\le n/2$, the admissible Newton polygon uniquely determined by $(n,r)$ is always $L$-hypersymmetric. Indeed, when $n=3r$ and $r$ is even, or when $n=4r$ and $r$ is odd, the Newton polygon consists of $3$ slopes of the same multiplicity and is hence $L$-balanced. Otherwise, the Newton polygon is an amalgamation of one polygon of slope $1/2$ and one polygon of slopes $\frac{1}{2}-\frac{1}{2r}, \frac{1}{2}+\frac{1}{2r}$. Clearly, each of these two polygons is $L$-balanced.

It is well-known that in this case, each isogeny class of Barsotti-Tate groups consists of a unique isomorphism class. Thus, every central leaf coincides with the Newton polygon containing it and therefore admits a $L$-hypersymmetric point. Moreover, by \cite[Theorem 1.1]{MR3240772}, every Newton stratum is irreducible, which then implies that every central leaf is irreducible. Following the same argument as the proof of Theorem \ref{mainthm}, the irreducibility of central leaves combined with Theorem \ref{cts} yields the desired result.
\end{proof}

\appendix
\section{Hypersymmetric Abelian varieties}

\hfill

This appendix reproduces Zong's proof of his main theorem \cite[Theorem 5.1]{zong2008hypersymmetric}, of which Theorem \ref{hs} is a rephrase in our simpler terminology. As mentioned in Section \ref{mu}, Theorem \ref{hs} follows from the proof but not the statement of \cite[Theorem 5.1]{zong2008hypersymmetric}. For completeness, we reproduce his proof to make the present paper self-contained. All the proofs in this appendix are due to Zong. Any mistake is my own. 

We remark that some of the conditions of supersingular restriction (see \cite[Section 4]{zong2008hypersymmetric}) come from the presence of a $B$-action (see in particular ``CM-type partitions'' in \cite[Proposition 4.9]{zong2008hypersymmetric}), or, in other words, from the requirement that the Newton strata under consideration is non-empty. Since we are only interested in non-empty Newton strata, we may omit those conditions to get a simpler formulation. The proofs in this appendix are essentially the same as those given in \cite[\textsection 6 and \textsection 7]{zong2008hypersymmetric}, except that the proof of Proposition A.4 is simpler than that of \cite[Proposition 7.7]{zong2008hypersymmetric} by omitting the discussion of CM type partitions. Theorem 3.2 follows from Proposition A.2 and A.3. 

\begin{customthm}{3.3}
A Newton stratum $\mathcal{N}$ contains a simple $B$-hypersymmetric point if and only if its Newton polygon is $B$-balanced.
\end{customthm}

\begin{notation}\cite[Definition 4.3.1]{zong2008hypersymmetric} Let $\zeta_B$ denote the Newton polygon with slope $1/2$ such that at every place $v$ of $F$ above $p,$ the multiplicity of $1/2$ is equal to the order of the class $[\mathbb{Q}_{p,\infty}\otimes F]-[B]$ in the Brauer group $Br(F),$ where $\mathbb{Q}_{p,\infty}$ is the quaternion $\mathbb{Q}$-algebra ramified exactly at $p$ and infinity. \end{notation}

The following proposition is a rephrase of \cite[Proposition 6.1]{zong2008hypersymmetric}.

\begin{prop}\label{prop1} Let $A$ be a simple $B$-linear polarized abelian variety over $\overline{\mathbb{F}_p}.$  If $A$ is $B$-hypersymmetric, then its Newton polygon $\zeta$ is either $\zeta_B$ or $B$-balanced.\end{prop}

\begin{proof}Let $A'$ be a simple $B$-linear polarized abelian variety over a finite field $\mathbb{F}_q$ such that $A'\otimes\overline{\mathbb{F}_p}\cong A$. We may assume $q$ is sufficiently divisible. $A'$ is $B$-simple and is isogenous to $A''$ for some absolutely simple abelian variety $A''$. Let $\pi$ denote the Frobenius of $A'$ and let $K=F(\pi).$ Write $D:=\End_B^0(A')$ and $L_v:=F_v\otimes K(\mathbb{F}_q).$ Write $M'$ for the $B$-linear isocrystal associated to $A'.$ Then $M'$ decomposes as $\oplus_{v|p}M_v$ and one checks that each $M_v$ is a free $L_v$-module. Let $f_v(T)$ be the characteristic polynomial of $\pi$ as an $L_v$-linear transformation of $M_v.$ Since $A$ is $B$-hypersymmetric, by \cite[Proposition 3.4]{zong2008hypersymmetric}, $$f_v(T)=\prod_{w|v}(T-\iota_w(\pi))^{n_w},$$ where $\iota_w:K\rightarrow F_v$ denotes the $F$-embeddings of $K$ into $F_v$ indexed by the place $w.$ By Kottwitz \cite{kottwitz1992points}, each $M_v$ has $[K:F]$ isotypic components.
Thus $f(T)=\det(T-\pi|_{M'})$ can be factored as $$f(T)=\prod_{v|p} Norm_{L_v/K(\mathbb{F}_q)}f_v(T)=\prod_{v|p}\prod_{w|v}Norm_{F_v/\mathbb{Q}_p}(T-\iota_w(\pi))^{n_w}.$$

Consider triples $u=(v,w,\tau)$ where $v$ is a place of $F$ above $p,$ $w$ is a place of $K$ above $v,$ and $\tau$ is an embedding $F_v\rightarrow\overline{Q}_p.$ Then the set of such triples is in $1$-to-$1$ correspondence with the set of embeddings $\iota_u:K\rightarrow \overline{Q}_p.$ Thus $$f(T)=\prod_u(T-\iota_u(\pi))^{n_w}.$$By Katz-Messing \cite{katz1974some}, $f(T)\in \mathbb{Z}[T],$ so $n_w$ is independent of $w.$ Hence, we have $n_w=2\dim(A)/[K:\mathbb{Q}].$ By \cite[Lemma 3.3]{kottwitz1992points}, the multiplicity of each isotypic component of $M$ is equal to $$n[L_v:K(\mathbb{F}_q)]/([B:F]^{1/2}[F_v:\mathbb{Q}_p])=order([D]).$$ In particular, if $\pi$ is totally real, $A'$ is a power of a supersingular elliptic curve with Newton polygon $\zeta_B,$ and the order of $[D]$ in $Br(F)$ is equal to $[\mathbb{Q}_{p,\infty}\otimes F]-[B].$ 
Finally, if $F=F_0$, then $[K:F]$ is even and each $\zeta_v$ is symmetric, so $\zeta$ contains no slope $1/2$ component. \end{proof}

The following proposition is a rephrase of the ``if'' part of \cite[Theorem 5.1]{zong2008hypersymmetric} (see \cite[Section 7]{zong2008hypersymmetric}). 

\begin{prop}\label{prop2}
Let $\zeta$ be a $B$-balanced Newton polygon such that the corresponding Newton stratum $\mathcal{N}$ is non-empty. Then there exists a hypersymmetric $B$-linear polarized simple abelian variety $A$ defined over $\overline{\mathbb{F}_p}$ whose Newton polygon is $\zeta.$ 
\end{prop}

\begin{proof}Case 1: Assume $F\neq F_0$. Since $\zeta$ is $B$-balanced, we may write $N=N_v$ for all places $v$ of $F$ above $p.$ For a place $v$ of $F$ above $p,$ write $v'$ for $v|_{F_0}.$ For each $v,$ define an $(F_0)_{v'}$-algebra of rank $N$ as follows:
\begin{equation}
T_{v'}=\begin{cases}
(F_0)_{v'}^N &\text{if }v\neq v^*\\
(F_0)_{v'}\times F_v^{(N-1)/2} &\text{if }v=v^*,N\text{ odd}\\
F_v^{N/2}&\text{if }v=v^*,N\text{ even}.
\end{cases} 
\end{equation}
By Lemma \ref{lem1}, there is a totally real extension $E/F_0$ of degree $N$ such that its normal hull has Galois group $S_N,$ and $E\otimes_F F_v\cong T_v'$ for all $v|p.$ 

Let $K=E\otimes_{F_0}F.$ Then the normal hull of $K/F$ has Galois group $S_N$ and $K\otimes_F F_v\cong F_v^N$. So for each $v$, we may index the slopes of $\zeta_v$ by the 
places of $w$ above $v$ in such a way that $\lambda_w+\lambda_{\overline{w}}=1.$ By Lemma \ref{lem3}, there is an integer $a\ge 1$ and a
$p^a$-Weil number $\pi$ such that $ord_w(\pi)/ord_w(p^a)=\lambda_w$ for all $w.$ We claim that $K=F(\pi).$ Indeed, if $N=1,$ $K=F=F(\pi).$ If $N>1,$ $\pi\notin F$ since $ord_{w}(\pi)$ 
and $ord_{w'}(\pi)$ are distinct for any $w\neq w'|p$ by construction, so $F\subsetneq F(\pi)\subseteq K$. By construction, $[K:F]=N.$ 
Then Lemma \ref{lem2} implies that $K=F(\pi).$ 

By Honda-Tate theory, there is a $B$-simple abelian variety $A'$ over $\mathbb{F}_{p^a}$corresponding 
to $\pi.$ Without loss of generality, we may assume $A'$ is absolutely $B$-simple. Let $A=A'\otimes_{\mathbb{F}_{p^a}}\overline{\mathbb{F}_p}.$ $A$ is hypersymmetric by Proposition \cite[Proposition 3.4]{zong2008hypersymmetric} and the Newton polygon of $A$ is $\zeta$.

Case 2: Assume $F=F_0.$ Since $\zeta$ is balanced, $N$ is even. By Lemma \ref{lem1}, there exists a totally real extension $E$ of $F$ of degree $N$ such that its normal hull has Galois group $S_N,$ and $E\otimes_F F_v\cong F_v^{N/2}$ for all $v|p.$ By Lemma 5.7 in \cite{chai2006hecke}, there is a totally imaginary quadratic extension $K/E$ such that $K\otimes_E E_v\cong E_v\times E_v$ for every $v|p$ and $K$ contains no proper CM extension of $F.$ Then, similar to the previous case, there is a hypersymmetric $B$-simple abelian variety $A$ whose Newton polygon is $\zeta.$ 
\end{proof}

\begin{lem}\cite[Proposition 7.2.3]{zong2008hypersymmetric}\label{lem1} Let $N$ be a positive integer and $F$ a totally real number field. Let $\Sigma$ be a finite set of non-archimedean places of $F.$ For each $v\in\Sigma$, let $T_v$ be a finite \'etale algebra over $F_v$ of rank $N.$ Then there is a totally real extension $E/F$ of degree $N$ such that its normal hull has Galois group $S_N,$ and $E\otimes_F F_v\cong T_v$ for all $v\in \Sigma.$\end{lem}

\begin{proof}Let $S=Spec\mathcal{O}_F,$ and let $X'=S[a_1,\cdots,a_N]$ be an $S$-affine space. Let $Y'\subseteq S[t]$ be given by $f=t^N+a_1t^{N-1}+\cdots +a_N.$ Let $R=R(f,f')$ be the resultant of $f$ and its derivative. Define $X=X'-\{R=0\}$ and $Y=Y'\times_{X'}X,$ then $Y$ is a \'etale cover of $X$ of rank $N.$  $X$ satisfies the property of weak approximation. The geometric fiber $Y_{\overline{K}}$ is affine of ring $\Gamma(\mathcal{O}_{Y_{\overline{K}}})=(\overline{K}[a_1,\cdots,a_N,t]/(f))_R.$ Write $A=\overline{K}[a_1,\cdots,a_N]$ and $B=\overline{K}[b_1,\cdots,b_N]$ where $b_i=a_i/a_n$ for $1\le i\le n-1$ and $b_n=a_n.$ Then $f$ is Eisenstein in the ring $B[t]$ with respect to the prime $a_n.$ Thus $f$ is irreducible in $B$ and hence also irreducible in $A.$ This implies that $\Gamma(\mathcal{O}_{Y_{\overline{K}}})$ is an integral domain.

Let $M$ denote the set of rational points $x$ of $X$ where $Y_x$ is connected. By Ekedahl's Hilbert irreducibility theorem (Theorem 1.3 in~\cite{ekedahl1988effective}), $M$ satisfies weak approximation. The requirement that $E\otimes_F F_v\cong T_v$ for all $v|p$ imposes a weak approximation condition on the parameters $a_i\in K.$ The condition on the Galois group of the normal hull is also a weak approximation property (see ~\cite{viehmann2013ekedahl}). The proposition follows by slightly modifying the content but not the proof of Ekedahl's theorem.\end{proof}

\begin{lem}\cite[Proposition 7.1.2]{zong2008hypersymmetric}\label{lem2} Let $N$ be a positive integer, $F$ a field, and $E/F$ a separable extension of degree $N.$ If the normal hull of $E/F$ has Galois group $S_N,$ then $E/F$ has no non-trivial sub-extension.\end{lem}

\begin{proof}$S_{N-1}$ is a maximal subgroup of $S_N.$\end{proof}

\begin{lem}\cite[Proposition 7.1.1]{zong2008hypersymmetric}\label{lem3} Let $K$ be a CM field. Given a set of rational numbers $\{\lambda_w\}_{w\in Spec\mathcal{O}_K,w|p}$ in $[0,1]$ such that $\lambda_{w}+\lambda_{\overline{w}}=1$ for all $w,$ there exists an integer $a\ge 1$ and a $p^a$-Weil number $\pi$ such that $ord_w(\pi)/ord_w(p^a)=\lambda_w$ for all $w.$\end{lem}
\begin{proof}Let $K_0$ be the maximally totally real subfield of $K.$ For any place $v$ of $E$ above $p,$ define $\lambda_v=\min\{\lambda_w:w\in Spec\mathcal{O}_K,w|v\}.$ Let $h$ be the ideal class number of $K.$ If $v$ splits in $K,$ let $a_v\in \mathcal{O}_K$ be a generator of the principal ideal $w^h,$ where $w|v$ is a place of $K;$ if $v$ is inert in $K$, let $a_v\in\mathcal{O}_K$ be a generator of $v^h.$ We have $$ (p\mathcal{O}_K)^h=\left(\prod_{v,v\text{ splits in }K}(w\overline{w})^{e(v|p)}\prod_{v,\text{ inert in }K}v^{e(v|p)}\right)^h=\prod_v(a_w\overline{a_w})^{e(v|p)}\prod_v a_v^{e(v|p)}\cdot u$$ for some unit $u$ of $\mathcal{O}_K.$ 

Now let $c$ be a sufficiently divisible positive integer and write $\lambda_v=m_v/(m_v+n_v)$ with $c=m_v+n_v$,$m_v,n_v\in\mathbb{Z}.$ Define $$\pi=\prod_v(a_w^{m_v}\overline{a_w}^{n_v})^{e(v|p)}\prod_v a_v^{ce(v|p)}/2\cdot u^{c/2}.$$ Then $\pi\overline{\pi}=p^{hc}$ and $\pi$ is the desired $p^{hc}$-Weil number.\end{proof}
\clearpage
\bibliographystyle{alpha}
\bibliography{bib}

\end{document}